\documentclass[12pt,a4paper,twoside,reqno]{amsart}

\usepackage[all,cmtip]{xy}
\usepackage{tikz}
\usepackage{amsmath,amscd}
\usepackage{amssymb,amsfonts,mathrsfs}
\usepackage[driver=pdftex,margin=3cm,heightrounded=true,centering]{geometry}
\usepackage{JK}
\usepackage[colorlinks=true,linkcolor=blue,citecolor=blue]{hyperref}

\usepackage[all]{xy}
\newtheorem{thm}{Theorem}[section]
\newtheorem{notation}[thm]{Notation}
\usepackage{graphicx}

\begin{document}
\title[]{A transformation rule for the index of commuting operators}

\author{Jens Kaad}
\address{Institut de Math\'ematiques de Jussieu,
Universit\'e de Paris VII,
175 rue du Chevaleret,
75013 Paris,
France}
\email{jenskaad@hotmail.com}

\author{Ryszard Nest}
\address{Department of Mathematical Sciences,
Universitetsparken 5,
DK-2100 Copenhagen {\O},
Denmark}
\email{rnest@math.ku.dk}
%
%
%
\thanks{The first author is supported by the Danish Carlsberg Foundation.}
\thanks{The second author is partially supported by the Danish National Research Foundation (DNRF) through the Centre for Symmetry and Deformation.}

\begin{abstract}
In the setting of several commuting operators on a Hilbert space one defines the notions of invertibility and Fredholmness in terms of the associated Koszul complex. The index problem then consists of computing the Euler characteristic of such a special type of Fredholm complex.
%

In this paper we investigate transformation rules for the index under the holomorphic functional calculus. We distinguish between two different types of index results:

1) A global index theorem which expresses the index in terms of the degree function of the "symbol" and the locally constant index function of the "coordinates".

2) A local index theorem which computes the Euler characteristic of a localized Koszul complex near a common zero of the "symbol".

Our results apply to the example of Toeplitz operators acting on both Bergman spaces over pseudoconvex domains and the Hardy space over the polydisc.

The local index theorem is fundamental for future investigations of determinants and torsion of Koszul complexes.
\end{abstract}

\maketitle
\tableofcontents
\section{Introduction}
The purpose of this paper is to prepare the stage for the study of determinants and torsion of complexes associated to $n$-tuples $A=(A_1,\ldots,A_n)$ of commuting operators on a Hilbert space $\sH$ (cf.  \cite{Kaa:JTS}). Our basic tool is the study of the localizations of $\sH$ as a module over the ring $\hoc$ of germs of holomorphic functions on the Taylor spectrum of $A$ (cf. \cite{Tay:JSC}). The case of two commuting Toeplitz operators $(T_f,T_g)$ has been studied by Carey and Pincus (\cite{CaPi:JTS}) and the formulas obtained there involve the Tate tame symbol, an expression where the exponents are multiplicities of zeroes of the functions $f$ and $g$ which are holomorphic in the interior of the unit disc. The general case is expected to involve local indices as defined below.

Our principal aim is thus to prove index theorems in the multivariable setup. This takes the form of a transformation rule for the index under Taylor's holomorphic functional calculus. Loosely speaking we start out with a commuting tuple of "variables" $A$ together with a holomorphic "symbol" $g : \T{Sp}(A) \to \cc^m$ on the Taylor spectrum. Under the condition that the commuting tuple $g(A)$ is Fredholm we distinguish between two different index results:

\begin{enumerate}
\item A global index theorem which expresses the Fredholm index of $g(A)$ in terms of the locally constant index function $\la \mapsto \T{Ind}(A-\la)$ associated with the coordinates $A$ and local degrees of the symbol $g$ near the set of zeroes.
\item A local index theorem which expresses the Euler characteristic of a localized Koszul complex at a common zero $\la$ for the symbol $g$ in terms of the local degree (or intersection multiplicity) $\T{deg}_\la(g)$ and the index $\T{Ind}(A - \la)$.
\end{enumerate}

Before we provide more details on the above results, we give some information on the multivariable holomorphic calculus.

Naturally associated to a commuting $n$-tuple $A=(A_1,\ldots,A_n)$ is a Koszul complex (see Definition \ref{d:koszulcomplex}):
$$
K_*(A,\sH)=(\sH\otimes \Lambda_*\cc^n,d_A)
$$
where the boundary operator is of the form $d_A=\sum_i A_i\otimes {\epsilon_i}^*$ and the $\epsilon_i^*$'s are interior multiplication operators with the standard basis vectors in $\cc^n$. $A$ is called invertible if $K_*(A,\sH)$ is contractible and it is called Fredholm if the homology of $K_*(A,\sH)$ is finite dimensional. When $A$ is Fredholm, the index is minus the Euler characteristic of the Koszul complex,
\[
\T{Ind}(A) := -\sum_i {(-1)}^i \T{Dim}_{\cc}(H_i(A,\sH)).
\]
In the case where $n=1$, these definitions coincide with the usual terminology. 

An analogue of the spectrum of a single operator was introduced by Taylor (see \cite{Tay:JSC}). The Taylor spectrum $\T{Sp}(A)$ of $A$ is defined as the set of $\lambda\in \cc^n$ such that $A - \la$ is not invertible. The Taylor spectrum is a compact and non-empty set.
%
%
%
The main property of this spectrum is the fact that it supports the multivariable holomorphic functional calculus, i. e. a continuous homomorphism
$$
\hoc \rightarrow \sL(\sH)
$$
from the germs of analytic functions on $\T{Sp}(A)$ to bounded operators on $\sH$ such that the coordinate functions $z_i$ are mapped to the operators $A_i$.
In particular, given a holomorphic map
$$
g=(g_1,\ldots,g_m):\T{Sp}(A)\rightarrow \cc^m ,
$$
one obtains a new commuting tuple $g(A)=(g_1(A),\ldots g_m(A))$. 
%

The global index theorem can now be stated as follows.

\begin{theorem}
Suppose that $ g(A)$ is Fredholm. In this case the set  $Z(g)=g^{-1}(\{0\})\cap \T{Sp}(A)$ is finite and
$$
\T{Ind}( g(A))=\begin{cases}
\sum_{\lambda \in Z(g)}\T{deg}_\lambda (g)\T{Ind}(A-\lambda )&\mbox{when }n=m\\
0&\mbox{when }n < m
\end{cases}
$$
where $\T{deg}_\lambda (g)$ is the intersection multiplicity of  $\{ g_1=\ldots =g_n=0\}$ at $\lambda$.
\end{theorem}

We remark that the global index theorem was already proved by Eschmeier and Putinar in \cite[Theorem 10.3.13]{EsPu:SDA}. Their techniques are however different from ours as they rely on the acyclicity properties of \v{C}ech complexes associated with Stein open coverings of the spectrum. See also \cite{Put:BFI} and \cite{Lev:FCO}. The type of localization which we apply is of a more algebraic nature.

The local indices can be defined using the localization procedure alluded to above. From now on we will suppose that $g(A)$ is Fredholm.

Let $\hoc$ denote the ring of germs of holomorphic functions on $\T{Sp}(A)$. The holomorphic functional calculus gives $\sH$ the structure of a $\hoc$-module. For each $\la \in \T{Sp}(A)$ we consider the localized module $\sH_\la$ at the prime ideal ${\mathfrak p}_\lambda =\{ f\in \hoc\mid f(\lambda )=0\}$. The local index at $\la$ is then defined as minus the Euler characteristic of the localized Koszul complex $K_*(g,\sH_\la)$,
\[
\T{Ind}_\la(g(A)) := -\sum_i (-1)^i \T{Dim}_{\cc}(H_i(g,\sH_\la)).
\]
Note that the homology groups are finite dimensional as a consequence of the Fredholmness of $g(A)$.

Another way of constructing the localized modules $H_*(g,\sH_\la)$ consists of decomposing the finite dimensional homology groups $H_*(g(A),\sH)$ into the generalized eigenspaces of the commuting tuple $H_*(A)$ induced by the action of the coordinate tuple $A$. See Proposition \ref{p:locspe}.

An important feature of the localized homology group $H(g,\sH_\la)$ is that it can be turned into a graded module over the ring $\C O_\la$ of power series convergent near $\la \in \T{Sp}(A)$. The associated homomorphism $\C O_\la \to \sL(H(g(A),\sH)(\la))$ makes the diagram
\[
\begin{CD}
\hoc @>>> \C O_\la  \\
@VVV @VVV \\
\sL\big(H(g(A),\sH_{\la})\big) @= \sL\big(H(g(A),\sH_{\la})\big)
\end{CD}
\]
commute. Here the upper horizontal map is the restriction homomorphism and the left vertical map is the homomorphism associated with the natural action of $\hoc$ on $H(g(A),\sH_{\la})$.

The relation between the local indices and the global index can then be described by the summation formula:
\[
\T{Ind}(g(A)) = \sum_{\la \in Z(g)} \T{Ind}_\la(g(A)).
\]

Our local index theorem can now be stated as follows.

\begin{theorem}
Suppose that $g(A)$ is Fredholm and that $g(\la) = 0$. Then the homology groups $H_*(g,\sH_\la)$ are finite dimensional and
\[
\T{Ind}_\la( g(A))= \T{deg}_\lambda (g)\T{Ind}(A-\lambda ),
\]
where $\T{deg}_\lambda (g)$ is the intersection multiplicity of  $\{ g_1=\ldots =g_n=0\}$ at $\lambda$.
\end{theorem}

The main step in the proof of the local index theorem is an analytic analogue of the algebraic localization procedure described above. In particular it implies the following reciprocity result.

\begin{theorem}
Suppose that $A$ and $B$ are two commuting tuples of the same length. Let $g : \T{Sp}(A) \cup \T{Sp}(B) \to \cc^m$ be holomorphic. Suppose that $g(A)$ and $g(B)$ are Fredholm and that the sets $Z(g) \cap \T{Sp}(A) \cap \T{Sp}_{\T{ess}}(B)$ and $Z(g) \cap \T{Sp}(B) \cap \T{Sp}_{\T{ess}}(A)$ are empty.
Then
\[
\sum_{\mu \in Z(g) \cap \T{Sp}(B)}\T{Ind}(\mu -A) \cd \T{Ind}_\mu(g(B))
= \sum_{\la \in Z(g) \cap \T{Sp}(A)} \T{Ind}_\la(g(A)) \cd \T{Ind}(\la - B).
\]
\end{theorem}

\medskip

The structure of this paper is as follows. 

In the next section we collected the general definitions and results connected with the multivariable holomorphic functional calculus and some results from homological algebra related to Koszul complexes (see subsection \ref{s:koszul}) which are useful later on.

In Section \ref{s:spectrum} we collected some basic results on the behaviour of Fredholm spectrum under holomorphic maps.

The algebraic localization results are proved in Section \ref{s:nilban}.

The local index theorems for regular zeroes of $g$ is proved in Section \ref{s:reg}.

A simple proof of the global index theorem is given in Section \ref{s:gloind} (see Theorem \ref{t:indthe}).

The analytic localization is the subject of Section \ref{s:respro}.

The local index theorem is proved in Section \ref{s:locind}.

\section{Preliminaries}
\noindent{\em Let $A = (A_1,\ldots,A_n)$ be a commuting tuple of linear operators on a vector space $V$ over the complex numbers $\cc$.}
\subsection{Koszul homology}

Let $\La(\cc^n) = \op_{k=1}^n \La_k(\cc^n)$ denote the exterior algebra over $\cc$ on $n$-generators $e_1,\ldots,e_n$. For a subset $I \su \{1,\ldots,n\}$ we use the notation $e_I := e_{i_1} \wlw e_{i_k} \in \La_k(\cc^n)$, $i_1 < \ldots < i_k$. We will give $\La(\cc^n)$ the structure of a Hilbert space in which the basis $\{e_I\}_{I \su \{1,\ldots,n\}}$ is orthonormal. For each $i \in \{1,\ldots,n\}$ we set
$$
\ep_i : \La(\cc^n) \to \La(\cc^n);\ e_I \mapsto e_i \we e_I .
$$
The adjoint of this exterior multiplication operator will be denoted by $\ep_i^* : \La(\cc^n) \to \La(\cc^n)$.

\begin{dfn}\label{d:koszulcomplex}
The \emph{Koszul complex} $ K_*(A,V)$ of $A$ on $V$ is the chain complex
\[
(K_*(A,V),d_A )
\]
where
\(
K_*(A,V) = V \ot \La_*(\cc^n) \mbox{ and } d_A := \sum_{i=1}^n A_i \ot \ep_i^* .
\)
We will denote the homology groups of $K_*(A,V)$ by $H_*(A,V)$  and refer to these vector spaces as the \emph{Koszul homology groups} of $A$.
We will set $H(A,V) := \op_{k=0}^n H_k(A,V)$.

\noindent The tuple $A$ is said to be {\em invertible} if $H(A,V)=\{ 0 \}$.
\end{dfn}


The next definition is fundamental to the present text.

\begin{dfn}
A commuting tuple $A$ is \emph{Fredholm} when the Koszul homology groups $H_*(A,V)$ are finite dimensional. Given a Fredholm commuting tuple $A$, its \emph{Fredholm index} is the integer
\[
\T{Ind}(A) := \sum_{k=0}^n (-1)^{k+1} \T{Dim}_{\cc}(H_k(A,V)).
\]
\end{dfn}

%
Recall that the \emph{mapping cone} $C(\ga,X)$ of a chain map $\ga : X \to X$ is the chain complex
\[
\begin{CD}
X_0 @<\ma{cc}{d & \ga}<< X_1 \op X_0 @<\ma{cc}{d & \ga \\ 0 & -d}<< \ldots @<\ma{cc}{d & \ga \\ 0 & -d}<< X_n \op X_{n-1} @<\ma{c}{\ga \\ -d}<< X_n.
\end{CD}
\]
The homology groups of the mapping cone will be denoted below by $H_*(\ga,X)$.

Suppose that $B\in \T{End}(V)$ commutes with $A_1,\ldots ,A_n$ and let $A \op B$ denote the commuting $(n+1)$-tuple $(A_1,\ldots ,A_n,B)$. Then $B$ defines a chain map $K_*(B) :K_*(A,V)\rightarrow K_*(A,V)$. For future reference let us state the following easy observation.

\begin{lemma}\label{l:koscon}
The Koszul complex $ K_*(A \op B,V)$ is naturally isomorphic to the mapping cone $C\big(B, K_*(A,V)\big)$ of $K_*(B)$. In particular if any of the $A_i$'s is invertible, then $K_*(A,V)$ is contractible (i. e. its homology is zero).
\end{lemma}
\begin{proof}

For each $k \in \{0,\ldots,n+1\}$ we set
\[
\begin{split}
& \al_k : C_k\big(B,K_*(A,V)\big) = K_k(A,V) \op K_{k-1}(A,V) \to K_k(A \op B,V) \\
& \al_k :
\big( (\xi \ot e_I), (\eta \ot e_J) \big) \mapsto \xi \ot e_I + \eta \ot e_{n+1} \we e_J
\end{split}
\]
It is not hard to verify that the collection $\{\al_k\}$ defines a chain isomorphism $\al : C(B,K_*(A,V)) \to K_*(A \op B,V)$.
\end{proof}

\subsection{Finite dimensional case}\label{s:spedec}
Let $A = (A_1,\ldots,A_n)$ be a commuting tuple of linear operators on a vector space $V$ of finite dimension over the complex numbers $\cc$. The classical Lie theorem says that there exists a basis for $V$ in which all $A_i$'s are simultaneously upper triangular. To be more precise, set, for each $\la \in \cc^n$,
\[
V(\la) := \bigcup_{i_1,\ldots,i_n \in \nn} \T{Ker}(A_1 - \la_1)^{i_1} \calca \T{Ker}(A_n - \la_n)^{i_n} \su V
\]
and $\si (A)=\{\la\mid V(\la )\neq \{0\} \}$. The following holds.

\begin{theorem}[Lie theorem]\label{t:spedec} The natural homomorphism
$
\op_{\la \in \si(A)} V(\la)  \rightarrow V
$
is bijective and, for each $\la \in \si(A)$, there exists a basis for $V(\la)$ such that each operator $A_i : V(\la) \to V(\la)$ is represented by an upper triangular matrix with $\la_i$ as the only diagonal entry. In particular, if $V(0)\neq \{0\} $, then it contains a common zero eigenvector for all $A_i$'s.
\end{theorem}

As a consequence of the Lie theorem we get the next proposition.

\begin{prop}\label{p:exieig} Suppose that $A$ is an $n$-tuple of commuting operators on a finite dimensional vector space $V$. Then
\[
H_k(A ,V) \cong H_k(A,V(0))
\]
for all $k \in \{0,\ldots,n\}$. Furthermore,  the homology group
$$
H_n(A,V(0)) \cong \cap_{i=1}^n \T{Ker}(A_i)
$$
is nonzero if and only if $V(0)\neq \{0\}$.
\end{prop}
\begin{proof}
By the above theorem, $H_* (A,V)\cong \oplus_{\la \in \si(A)} H_* (A,V(\la))$. Suppose now that $\la \in \si (A)$ and $\la\neq 0$, say $\la_i\neq 0$. The operator $A_i : V(\la) \to V(\la)$ is then invertible and the Koszul complex $ K_*(A,V(\la))$ is therefore contractible. This proves the first part of the proposition.

The second part of the proposition follows from the fact that, if $V(0)$ is nonzero, then it contains a vector $\xi\in \cap_{i=1}^n \T{Ker}(A_i)$.
\end{proof}

\subsection{Taylor spectrum}$\mbox{ }$

\bigskip

\noindent{\em From now on we will suppose that $V$ has the additional structure of a Banach space and that the commuting linear operators $A_1,\ldots,A_n : V \to V$ are bounded.}

 \bigskip

 \noindent The Koszul complex can be used to define a good notion of joint spectrum of the n-tuple  $A$, often referred to as the \emph{Taylor spectrum}. See \cite{Tay:JSC}.

\begin{dfn}
The (Taylor)  \emph{spectrum} of $A$ is the set
\[
\T{Sp}(A) := \big\{\la \in \cc^n \, | \, H(A - \la,V) \neq \{0\} \big\}
\]
where $A - \la := (A_1 - \la_1,\ldots,A_n - \la_n)$.
The Taylor spectrum is a compact non-empty subset of $\cc^n$, \cite[Theorem 3.1]{Tay:JSC}.
\end{dfn}

As an immediate corollary of Proposition \ref{p:exieig} we get the following.

\begin{cor}\label{c:joispe} Suppose that $V$ is finite dimensional. Then the following are equivalent:
\begin{enumerate}
\item
$\la \in \T{Sp}(A)$
\item $\la \in \si(A) $
\item $H_n(A - \la,V) \neq \{0\}$.
\end{enumerate}
\end{cor}

The main result about the Taylor spectrum is the existence of a holomorphic functional calculus. In order to formulate this result we need to introduce some notation.

\begin{notation}
Let $\hoc$ denote the ring of germs of analytic functions on $\T{Sp}(A)$, i.e.  analytic functions $f : U \to \cc$ defined on some open set $U \ssu \T{Sp}(A)$ subject to the equivalence relation:
\vspace{5pt}

 \noindent  {\em $f : U \to \cc$ and $g : V \to \cc$ represent the same class in $\hoc$ if they agree on some open set $W$ with $\T{Sp}(A) \su W \su U \cap V$.}
\vspace{5pt}

 \noindent The ring structure is given by the pointwise  sum and product.

The notation $\C O(U)$ refers to the unital ring of holomorphic functions on some domain $U \su \cc^n$.
\end{notation}

For any subalgebra $\sB \su \sL(V)$ of the bounded operators on $V$ its commutant is the algebra
$$
\sB' = \{C \in \sL(V) \, | \, [C,B] = 0 \, , \, \T{ for all } B \in \sB \}
$$.

\begin{theorem}\cite[Theorem 4.8]{Tay:ACS}\label{t:anafun}
Let $\sA \su \sL(V)$ denote the smallest unital $\cc$-algebra which contains the bounded operators $A_1,\ldots,A_n$. There exists a unital homomorphism $\hoc \to \sA''$, $f \mapsto f(A)$ such that $z_i \mapsto A_i$. Furthermore, whenever $f = (f_1,\ldots,f_m) : \T{Sp}(A) \to \cc^m$ is analytic (i. e. $f_k\in \hoc ,\ k=1,\ldots,m$) the following identity holds:
$$
\T{Sp}(f(A)) = f\big(\T{Sp}(A) \big).
$$
\end{theorem}

A uniqueness result for the holomorphic functional calculus is contained in the following. 

\begin{theorem}\cite[Theorem 5.2.4]{EsPu:SDA}\label{t:anauni}
Let $U \ssu \T{Sp}(A)$ be an open set and let $\Phi : \C O(U) \to \sL(V)$ be a unital homomorphism. Suppose that $\Phi(f) = f(A)$ for all $f \in \C O(\cc^n)$ and that $h(\T{Sp}(A)) = \T{Sp}(\Phi(h))$ for each holomorphic function $h : U \to \cc^k$, $k \geq 1$. We then have the identity $\Phi(h) = h(A)$ for all $h \in \C O(U)$.
\end{theorem}

The holomorphic functional calculus is functorial with respect to $V$, in fact the following holds.

\begin{prop}\cite[Proposition 4.5]{Tay:ACS}\label{p:functoriality}
Let $V$ and $W$ be Banach spaces, $\Phi:V\rightarrow W$ be a continuous linear map. Suppose that  $A$ (resp. $B$) are commuting $n$-tuples on $V$ (resp. $W$) satisfying $B_i\Phi =\Phi A_i,\ i=1,\ldots n$. Then
$$
\Phi f(A)=f(B)\Phi .
$$
for all $f\in {\mathcal O}(\T{Sp}(A)\cup \T{Sp}(B))$.
\end{prop}

\begin{dfn} In the context of the above theorem we will endow $V$ with the structure of a module over $\hoc$ induced by the homomorphism $f \mapsto f(A)$.
\end{dfn}

\noindent{\em For the rest of the section, suppose that $V$ is a Hilbert space over $\cc$. }

\bigskip

Let $\sC(V)=\sL(V)/\sK(V)$ be the Calkin algebra, the quotient of the C*-algebra of bounded operators on $V$ by the ideal of compact operators. The commuting tuple $\pi(A)$ of bounded operators on $\sC(V)$ is given by the action of the $A_i$'s by left multiplication.

\begin{dfn}
The \emph{essential spectrum} of the commuting tuple $A$ is the Taylor spectrum of $\pi(A)$. It will be denoted by $\T{Sp}_{\T{ess}}(A)$.
\end{dfn}

The relation between the Fredholmness of $A$ and the essential spectrum of $A$ was clarified by Curto in \cite{Cur:FOD}. We state the result as a theorem.

\begin{theorem}{\cite[Corollary 6.2]{Cur:FOD}}\label{t:essspe}
Let $\la \in \cc^n$. The commuting tuple $A -\la$ is Fredholm if and only if $\la \notin \T{Sp}_{\T{ess}}(A)$.
\end{theorem}

Let $g = (g_1,\ldots,g_m) : \T{Sp}(A) \to \cc^m$ be an analytic map. In analogy with Theorem \ref{t:anafun} we can characterize the Fredholmness of $g(A)$ in terms of the image set $g(\T{Sp}_{\T{ess}}(A)) \su \cc^m$. See \cite{Lev:TSC} and \cite{Fai:JSL}.

\begin{prop}
The commuting tuple $g(A)$ is Fredholm if and only if $0 \notin g(\T{Sp}_{\T{ess}}(A))$.
\end{prop}

\subsection{A spectral sequence}\label{s:koszul}
A generalization of the mapping cone observation in Lemma \ref{l:koscon} is as follows.
\bigskip

\noindent \emph{Let $A = (A_1,\ldots,A_n)$ and $B = (B_1,\ldots,B_m)$ be commuting tuples of linear operators on a vector space $V$ such that the union $A \op B := (A_1,\ldots,A_n,B_1,\ldots,B_m)$ is a commuting $(n+m)$-tuple.}
\bigskip

Let us define a bigrading of the exterior algebra $\La(\cc^{n+m})$ which reflects the position of $A$ and $B$ in the union $A \op B$. For each pair $p \in \{0,1,\ldots,n\}$ and $q \in \{0,1,\ldots,m\}$ we define the subspace
\[
\begin{split}
& \La_{p,q}(\cc^{n+m}) \\
& \q := \T{span}\big\{e_I \we e_J \, | \, I \su \{1,\ldots,n\} \, , \, |I| = p \, , \, J \su \{n+1,\ldots,n+m\} \, , \, |J| = q\big\}.
\end{split}
\]
where we recall that $e_I := e_{i_1} \wlw e_{i_k}$ for any subset $I = \{i_1,\ldots,i_k\}\su \{1,\ldots,n+m\}$ with $i_1 < \ldots < i_k$. We then have an isomorphism $\La(\cc^{n+m}) \cong \op_{p = 0}^n \op_{q = 0}^m \La_{p,q}(\cc^{n+m})$ and this decomposition turns $\La(\cc^{n+m})$ into a bigraded algebra.

The bigrading of the exterior algebra leads to a bigrading of the Koszul chains $ K_*(A \op B,V) = V \ot \La(\cc^{n+m})$ by defining $K_{p,q}(A \op B,V) := V \ot \La_{p,q}(\cc^{n+m})$. This allows us to view the Koszul complex as the totalization of the following  bicomplex. We define the vertical differential $d^v : K_{p,q}(A \op B,V) \to K_{p-1,q}(A \op B,V)$ by $d^v :=  \sum_{i=1}^n A_i \ot \ep_i^*$ and the horizontal differential by $d^h : K_{p,q}(A \op B,V) \to K_{p,q-1}(A \op B,V)$, $d^h := \sum_{i=1}^m B_i \ot \ep_{i+n}^*$. Since $(d^v)^2 = (d^h)^2 = d^v d^h + d^h d^v = 0$ we get a bicomplex
\begin{equation}\label{eq:koscom}
\begin{CD}
K_{n,0}(A \op B) @<{d^h}<< K_{n,1}(A \op B) @<{d^h}<< \ldots @<{d^h}<< K_{n,m}(A \op B) \\
@V{d^v}VV @V{d^v}VV & & @V{d^v}VV \\
K_{n-1,0}(A \op B) @<{d^h}<< K_{n-1,1}(A \op B) @<{d^h}<< \ldots @<{d^h}<< K_{n-1,m}(A \op B) \\
@V{d^v}VV @V{d^v}VV & & @V{d^v}VV \\
\vdots & & \vdots & & & & \vdots \\
@V{d^v}VV @V{d^v}VV & & @V{d^v}VV \\
K_{0,0}(A \op B) @<{d^h}<< K_{0,1}(A \op B) @<{d^h}<< \ldots @<{d^h}<< K_{0,m}(A \op B).
\end{CD}
\end{equation}
It is not hard to see that the totalization of this bicomplex is isomorphic to our original Koszul complex $ K_*(A \op B,V)$.

Below we will describe  the homology spectral sequence associated with the filtration by rows. For each $i \in \{0,\ldots,n\}$ we define the sub-bicomplex $F_i$ consisting of the rows with indices $0,\ldots,i$ of the bicomplex \eqref{eq:koscom}. This gives a filtration
\[
0 \su F_0 \su F_1 \su \ldots \su F_{n-1} \su F_n = K_{**}(A \op B,V)
\]
with an associated spectral sequence converging to the Koszul homology of $A \op B$ (see f. ex. \cite[Theorem 5.5.1]{Wei:IHA}).
%

\begin{prop}\label{p:filcon}
The homology spectral sequence associated with the row filtration of the bicomplex $K_{**}(A \op B,V)$ converges to the Koszul homology of $A \op B$. The $E^2$-term of this spectral sequence is given by $E^2_{pq} = H_p\big(A , H_q(B,V) \big)$.
\end{prop}
\begin{proof}
Let $p \in \{0,\ldots,n\}$ and $q \in \{0,\ldots,m\}$. By definition the $E^1_{pq}$-term is given by the $q^{\T{th}}$ homology group of the chain complex $F_p/F_{p-1}$. The chain complex $F_p/F_{p-1}$ is given by the $p^{\T{th}}$ row of the bicomplex $K_{**}(A \op B,V)$ which is isomorphic to the chain-complex $\big(K_*(B,V) \ot \La_p(\cc^n), (-1)^p \cd d_B \ot 1\big)$. The term $E^1_{pq}$ of our spectral sequence is thus given by $H_q(B,V) \ot \La_p(\cc^n)$. The differential $d^1 : E^1_{pq} \to E^1_{p-1,q}$ is nothing but the Koszul-differential of the commuting tuple $H_q(A) := \big( H_q(A_1),\ldots,H_q(A_n) \big)$ which acts on the homology group $H_q(B,V)$. It follows that the $E^2_{pq}$-term is given by the homology group $H_p\big( A, H_q(B,V) \big)$ as desired.
\end{proof}

In the remainder of this section we will prove corollaries which will be needed later on.

\begin{prop}\label{p:nontri}
Suppose that $B = (B_1,\ldots,B_m)$ is Fredholm and that the Koszul homology group $H_p(A,H_q(B,V))$ is non-trivial for some $p \in \{0,\ldots,n\}$ and $q \in \{0,\ldots,m\}$. Then there exists a $k \geq p+q$ such that $H_k(A \op B, V)$ is non-trivial as well.
\end{prop}
\begin{proof}
Let $(E^r,d^r)$ denote the homology spectral sequence associated with the row filtration of the bicomplex $K_{**}(A \op B,V)$. Recall that the differential $d^r$ sends $E^r_{pq}$ to $E^r_{p-r,q+r-1}$.
%

Since $B$ is Fredholm by assumption, $H_q(B,V)$ is finite dimensional. By Corollary \ref{c:joispe} the non-triviality of the homology group $H_p(A,H_q(B,V))$ implies that the homology group $H_n(A,H_q(B,V))$ is non-trivial as well.

To continue, note that the homology group $E^{r+1}_{nq}$ can be identified with the kernel of the differential $d^r : E^r_{nq} \to E^r_{n-r,q+r-1}$ for all $r \in \nn$. By Proposition \ref{p:filcon}  $H_n(A,H_q(B,V)) \cong E^2_{nq}$. The above reasoning therefore gives us a non-trivial vector $\xi \in E^2_{nq}$. Suppose now that $\xi \in E^2_{nq}$ determines a class in $E^r_{nq}$ for all $r \geq 2$ thus that $d^r(\xi) = 0$ for all $r \geq 2$. This implies that the limit $E^\infty_{nq}$ is non-trivial and hence, by the convergence of the spectral sequence, that the homology group $H_{n+q}(A \op B,V)$ is non-trivial. Since $n+q \geq p+q$ we can thus assume, without loss of generality, that $d^r(\xi) \neq 0$ for some $r \geq 2$.

This assumption implies in particular that the homology group $E^r_{n-r,q + r-1}$ is non-trivial. But this can only happen if the homology group $E^2_{n-r,q + r -1} = H_{n-r}(A, H_{q + r - 1}(B,V))$ is non-trivial as well. By the Fredholmness assumption on $B$ the homology group $H_n(A,H_j(B,V))$, $j = q + r-1 > q$, is non-zero. Applying the same argument a finite number of times we may assume, without loss of generality, that the homology group $H_n(A,H_m(B,V))$ is non-trivial. But this group agrees with the $E^{\infty}_{nm}$-term of the spectral sequence and hence $H_{n+m}(A \op B,V)$ is non-trivial. This proves the claim of the proposition.
\end{proof}

\begin{prop}\label{p:kostri}
Let $k \in \{0,\ldots,n+m\}$ and suppose that the homology group $H_p(A,H_q(B,V))$ is trivial for all $p \in \{0,\ldots,n\}$ and all $q \in \{0,\ldots,m\}$ with $p+q = k$. Then the homology group $H_k(A \op B,V)$ is trivial.
\end{prop}
\begin{proof}
Let $(E^r,d^r)$ denote as above the homology spectral sequence associated with the row filtration of the bicomplex $K_{**}(A \op B,V)$. By Proposition \ref{p:filcon}  the $E^2$-term is given by $E^2_{pq} := H_p(A,H_q(B,V))$. The assumptions imply that $E^2_{pq} = 0$ for all $p,q$ with $p+q = k$. The convergence of the spectral sequence implies that $H_k(A \op B,V)$ is trivial as well.
\end{proof}

\begin{prop}\label{p:speind}
Suppose that the homology group $H_p(A,H_q(B,V))$ is finite dimensional for all $p \in \{0,\ldots,n\}$ and $q \in \{0,\ldots,m\}$. Then the commuting tuple $A \op B$ is Fredholm and the index is given by
\[
\T{Ind}(A \op B) = \sum_{p,q} (-1)^{p+q+1}\T{Dim}\big(  H_p(A,H_q(B,V)) \big).
\]
\end{prop}
\begin{proof}
By an application of Proposition \ref{p:filcon} we see that the $E^2$-term of the homology spectral sequence associated with the row filtration of $K_{**}(A \op B,V)$ is finite dimensional. Since each term of the spectral sequence is obtained by taking homology groups of its predecessor we get that the $E^r$-term is finite dimensional for all $r \geq 2$. The convergence of the spectral sequence then implies that $\T{Dim}_{\cc}H_k(A \op B,V) = \sum_{p+q = k} \T{Dim}_{\cc}(E^\infty_{pq})$ is finite for each $k \in \{0,\ldots,n+m\}$. This means that $A \op B$ is Fredholm. Furthermore, we see that
\[
\T{Ind}(A \op B) = \sum_{p,q} (-1)^{p+q+1} \T{Dim}_{\cc}(E^\infty_{pq}).
\]
Using one more time that each term of the spectral sequence is the homology of its predecessor together with general facts about Euler characteristics we get that
\[
\sum_{p,q} (-1)^{p+q +1} \T{Dim}_{\cc}(E^\infty_{pq}) = \sum_{p,q} (-1)^{p+q+1} \T{Dim}_{\cc}(E^r_{pq})
\]
for all $r \geq 2$. This proves the claim of the proposition since $E^2_{pq} = H_p(A,H_q(B,V))$.
\end{proof}

\section{Spectrum of Fredholm tuples}\label{s:spectrum}

 \begin{notation}Let $A = (A_1,\ldots,A_n)$ be a commuting tuple of bounded operators on a Hilbert space $\sH$, and let $g : \T{Sp}(A) \to \cc^m$ be a holomorphic map.
 \begin{itemize}
 \item $g(A) = (g_1(A),\ldots,g_m(A))$ denotes the commuting tuple obtained from $A$ and $g$ by the holomorphic functional calculus.
 \item Since each of the operators $g_j(A)$ commute with each of the operators $A_i$, the $A_i$'s induce a commuting tuple of linear operators on the Koszul homology groups $H_*(g(A),\sH)$. We will denote this n-tuple by
     $$
     H_*(A) := (H_*(A_1),\ldots,H_*(A_n)).
     $$
\item $Z(g) := \{\la \in \T{Sp}(A)\, | \, g_1(\la) = \ldots = g_m(\la) = 0\}$ is the set of common zeroes.
 \end{itemize}
When $g(A)$ is Fredholm, the Taylor spectum $\T{Sp}(H_k(A))$ makes sense and is a finite set for each $k \in \{0,\ldots,m\}$. In this case let
\[
\T{Sp}(H(A)) := \cup_{k=0}^m \T{Sp}(H_k(A)).
\]
\end{notation}

\begin{lemma}\label{l:zerspe}
The set of common zeroes $Z(g)$ agrees with the set $\big\{ \la \in \cc^n \, | \, (\la,0) \in \T{Sp}(A \op g(A))\big\}$.
\end{lemma}
\begin{proof}
In fact, by Theorem \ref{t:anafun}, $\T{Sp}(A \op g(A))$ coincides with the graph of the map $g : \T{Sp}(A) \to \cc^m$.
\end{proof}

\begin{theorem}\label{t:equzerspe}
Suppose that $g(A)$ is Fredholm. Then
$$
Z(g) = \T{Sp}({H}(A)).
$$
In particular, the set of common zeroes for $g$ is finite.
\end{theorem}
\begin{proof}
The result of the theorem follows from the following bi-implications:
\[
\begin{split}
\la \in Z(g)
& \Leftrightarrow H( (A-\la) \op g(A), \sH) \neq \{0\}
\Leftrightarrow H\big(A-\la,H(g(A),\sH)\big) \neq \{0\} \\
& \Leftrightarrow \la \in \T{Sp}({H}(A)).
\end{split}
\]
The first bi-implication is the statement of Lemma \ref{l:zerspe} and the second one follows from Proposition \ref{p:nontri} and Proposition \ref{p:kostri}. The last bi-implication follows by definition of the set $\T{Sp}({H}(A))$.
\end{proof}

\begin{notation}For each zero $\la \in Z(g)$ and each $k \in \{0,\ldots,m\}$, $H_k(g(A),\sH)(\la) \su H_k(g(A),\sH)$ denotes the finite dimensional vector space coming from the spectral decomposition of $H_k(g(A),\sH)$ with respect to the commuting tuple ${H}_k(A)$ (see Section \ref{s:spedec}).
\end{notation}

The above spectral decomposition of the Koszul homology groups allows us to define a local version of the Fredholm index.

\begin{dfn}\label{d:locind}
Suppose that $g(A)$ is Fredholm and that $\la \in Z(g)$. The \emph{local index} of $g(A)$ at $\la$ is the integer
\[
\T{Ind}_\la(g(A)) := \sum_{k=0}^m (-1)^{k+1} \T{Dim}_{\cc}\big( H_k(g(A),\sH)(\la) \big).
\]
\end{dfn}

The relation between the local indices for $g(A)$ and the global index for $g(A)$ is given by the following:

\begin{prop}\label{p:sumind}
Suppose that $g(A)$ is Fredholm. Then the index of $g(A)$ can be computed as the sum of the local indices. Thus,
\[
\T{Ind}(g(A)) = \sum_{\la \in Z(g)} \T{Ind}_\la(g(A)).
\]
\end{prop}
\begin{proof}
Let $k \in \{0,\ldots,m\}$. By Theorem \ref{t:spedec}, Corollary \ref{c:joispe}, and Theorem \ref{t:equzerspe} we have the following isomorphisms
\[
\begin{split}
H_k(g(A),\sH) & \cong \bop_{\la \in \si({H}_k(A))} H_k(g(A),\sH)(\la)
\cong \bop_{\la \in \T{Sp}({H}_k(A))} H_k(g(A),\sH)(\la) \\
& \cong \bop_{\la \in Z(g)} H_k(g(A),\sH)(\la)
\end{split}
\]
where we recall that $H_k(g(A),\sH)(\la) = \{0\}$ whenever $\la \notin \si({H}_k(A))$. This immediately implies the result of the proposition.
\end{proof}

\section{Algebraic localization}\label{s:nilban}
Let $\sH$ be a Hilbert space and let $A = (A_1,\ldots,A_n)$ be a commuting tuple of bounded operators on $\sH$. Recall that the analytic functional calculus (Theorem \ref{t:anafun}) gives $\sH$ the structure of a unital $\hoc$-module.

Throughout this section the following assumption will be in use.
\bigskip

\noindent \emph{ Let $g : \T{Sp}(A) \to \cc^m$ is a holomorphic map such that $g(A) := (g_1(A),\ldots,g_m(A))$ is Fredholm. Let $\la \in \T{Sp}(A)$.}

\bigskip

By the results in Section \ref{s:spectrum} there is a decomposition
\[
H(g(A),\sH) \cong  \bop_{\mu \in Z(g)} H(g(A),\sH)(\mu)
\]
where each $H(g(A),\sH)(\mu) \su H(g(A),\sH)$ is a generalized eigenspace for the commuting tuple $H(A) := (H(A_1),\ldots,H(A_n))$.

\begin{notation}
For each $f \in \hoc$, let 
\[
H(f(A))(\la) : H(g(A),\sH)(\la) \to H(g(A),\sH)(\la)
\]
denote the associated endomorphism of the generalized eigenspace. In particular we have the commuting tuple
\[
H(A)(\la) := \big(H(A_1)(\la),\ldots,H(A_n)(\la)\big).
\]
\end{notation}

Let $f \in \hoc$. Since $H(A_i)(\la) - \la_i$ is nilpotent for each $i \in \{1,\ldots,n\}$. If $g(\la) = 0$ we thus get that $\T{Sp}\big(H(A)(\la)\big) = \{\la\}$. The analytic functional calculus therefore yields a linear operator 
\[
f(H(A)(\la)) : H(g(A),\sH)(\la) \to H(g(A),\sH)(\la).
\]
If $g(\la) \neq 0$, $H(g(A),\sH)(\la) = \{0\}$.

The first aim of this section is to prove the identity
\begin{equation}\label{eq:cohcom}
f\big( H(A)(\la) \big) = H(f(A))(\la).
\end{equation}

\begin{lemma}\label{l:cohcom}
For each $f \in \C O(\cc^n)$ we have the identity
\[
f\big( H(A)(\la) \big) = H\big(f(A)\big)(\la)
\]
of endomorphisms of the generalized eigenspace $H\big(g(A),\sH\big)(\la)$.
\end{lemma}
\begin{proof} Let $d_{g(A)}$ denote the differential in the Koszul complex computing $H\big(g(A),\sH\big)$, let $q$ denote the quotient map $\T{Ker}(d_{g(A)})\rightarrow H\big(g(A),\sH\big)$ and let $\io : \T{Ker}(d_{g(A)}) \to \sH \ot \La(\cc^m)$ denote the inclusion. Both the inclusion and the quotient map are continuous and we have the identities $A_i \io = \io A_i|_{\T{Ker}(d_{g(A)})}$ and $q A_i|_{\T{Ker}(d_{g(A)})} = H(A_i) q$ for all $i \in \{1,\ldots,n\}$. Since $f \in \C O(\cc^n)$ we can conclude from Proposition \ref{p:functoriality} that
\[
f(H(A)) = H(f(A)). 
\]
This proves the claim of the lemma since
\[
\op_{\mu \in Z(g)} f\big(H(A)(\mu)\big) = f(H(A)) = H(f(A)) = \op_{\mu \in Z(g)} H(f(A))(\mu).
\]
\end{proof}

The next lemma allows us to compute the spectrum of endomorphisms of the form $H(f(A))(\la)$.

\begin{lemma}\label{l:hominv}
Suppose that $f : \T{Sp}(A) \to \cc$ is holomorphic and $f(\la) \neq 0$. The endomorphism $H(f(A))(\la)$ is then invertible.
\end{lemma}
\begin{proof}

The mapping cone construction of the Koszul complex (Lemma \ref{l:koscon}) yields a long exact sequence
\[
\begin{CD}
\ldots @>>> H_{k+1}(f \op g,\sH) @>>> H_k(g,\sH) @>{H_k(f(A))}>>
H_k(g,\sH) \\
& & & & & & @VVV \\
& &  \ldots @<<< H_{k-1}(f \op g,\sH) @<<< H_k(f \op g,\sH) 
\end{CD}
\]
of homology groups. It can be verified that each of the linear maps in this sequence intertwines the action of the commuting tuple $A$. They are in fact homorphisms for the $\C O(\T{Sp}(A))$-module structure on the involved homology groups. See \cite{Kaa:JTS} for an explicit description of these maps at the level of complexes. In particular we get a long exact sequence 
\[
\begin{CD}
\ldots @>>> H_{k+1}(f\op g,\sH)(\mu) @>>> H_k(g,\sH)(\mu) @>{H_k(f)}>>
H_k(g,\sH)(\mu) \\
& & & & & & @VVV \\
& &  \ldots @<<< H_{k-1}(f \op g,\sH)(\mu) @<<< H_k(f \op g,\sH)(\mu) 
\end{CD} 
\]
for each $\mu \in Z(g)$. It is therefore enough to show that the generalized eigenspace $H(f(A) \op g(A),\sH)(\la)$ is trivial. This is a consequence of Theorem \ref{t:equzerspe} since $\la \notin Z(f) \cap Z(g)$.
\end{proof} 

The uniqueness result for the holomorphic functional calculus (Theorem \ref{t:anauni}) now allows us to prove the identity in Equation \eqref{eq:cohcom}.

\begin{prop}\label{p:cohcom}
For each $f \in \C O(\T{Sp}(A))$ we have the identity
\[
f\big( H(A)(\la) \big) = H\big(f(A)\big)(\la)
\]
of endomorphisms of the generalized eigenspace $H\big(g(A),\sH\big)(\la)$.
\end{prop}
\begin{proof}
Let $\Phi : \C O(\T{Sp}(A)) \to \sL\big( H(g(A),\sH) \big)$ denote the unital homomorphism $\Phi(f) := H(f(A))$. We need to show that $\Phi(f) = f(H(A))$ for all $f \in \hoc$.

By Theorem \ref{t:anauni} we only need to prove that $\Phi(f) = f(H(A))$ for all $f \in \C O(\cc^n)$ and that $h\big( \T{Sp}(H(A)) \big) = \T{Sp}(\Phi(h))$ for all holomorphic functions $h : \T{Sp}(A) \to \cc^k$, $k \geq 1$.

The first of these assertions follows immediately from Lemma \ref{l:cohcom}.

Dealing with the second assertion amounts to proving that $h\big( Z(g) \big) = \T{Sp}(\Phi(h))$. Note first that $\T{Sp}(\Phi(h)) = \cup_{\mu \in Z(g)}\T{Sp}\big( H(h(A))(\mu) \big)$. This is a consequence of our spectral decomposition of $H(g(A),\sH)$. It is therefore enough to show that the commuting tuple $H(h(A))(\mu) - \eta$ is invertible if and only if $\eta \neq h(\mu)$. But this is an easy consequence of Lemma \ref{l:hominv} and the fact that the Taylor spectrum is non-empty.
\end{proof}

\begin{notation}
For each $\la \in \cc^n$, let $\C O_\la$ denote the stalk of the sheaf of analytic functions on $\cc^n$ at $\la$. This stalk can be identified with the unital commutative ring of convergent power series near $\la$. When $\la \in \T{Sp}(A)$ there is a well-defined restriction map $\hoc \to \C O_\la$.
\end{notation}

As an application of the above result we have the following useful:

\begin{prop}\label{p:powfac}
Suppose that $g(A) = (g_1(A),\ldots,g_m(A))$ is Fredholm and let $\la \in \T{Sp}(A)$. Then the generalized eigenspace $H(g(A),\sH)(\la)$ can be turned into a graded module over the ring $\C O_\la$ of convergent power series in such a way that the associated homomorphism $\C O_\la \to \sL(H(g(A),\sH)(\la))$ makes the diagram
\[
\begin{CD}
\hoc @>>> \C O_\la  \\
@VVV @VVV \\
\sL\big(H(g(A),\sH)(\la)\big) @= \sL\big(H(g(A),\sH)(\la)\big)
\end{CD}
\]
commute. Here the upper horizontal map is the restriction homomorphism and the left vertical map is the homomorphism associated with the action of $\hoc$ on $H(g(A),\sH)(\la)$.
\end{prop}
\begin{proof}
The existence of a unital graded homomorphism $\C O_\la \to \sL\big(H(g(A),\sH)(\la)\big)$ which maps the coordinates $z_i -\la_i$ to $H(A_i - \la_i)(\la)$ is a consequence of the analytic functional calculus. Indeed, we have that $\T{Sp}(H(A)(\la)) = \{\la\}$ when $\la \in Z(g)$. The statement is trivial otherwise.

In order to prove the lemma we therefore only need to verify the identity $H(f(A))(\la) = f({H}(A)(\la))$ for each $f \in \hoc$. But this is the content of Proposition \ref{p:cohcom}.
%
\end{proof}

\begin{notation}
For each $\la \in \T{Sp}(A)$, let $\sH_\la$ denote the localization of the module $\sH$ w.r.t. the prime ideal $\G p_\la =\{f \in \hoc\mid f(\la) = 0\}$.  The localization $\sH_\la$ remains a module over $\hoc$ with action $f \cd (\xi/h) := (f(A) \xi)/h$.
\end{notation}

Our results allow us to describe the generalized eigenspace $H(g(A),\sH)(\la)$ using the Koszul homology groups $H(g,\sH_\la)$.

\begin{prop}\label{p:locspe}
Suppose that $g(A)$ is Fredholm and that $\la \in \T{Sp}(A)$. Then there exists an isomorphism $H(g,\sH_\la) \cong H(g(A),\sH)(\la)$ of graded modules.
\end{prop}
\begin{proof}
We start by noting that $H(g(A),\sH_\la)$ is isomorphic to $H(g(A),\sH)_\la$ where as above $H(g(A),\sH)_\la$ denotes the localization of the module $H(g(A),\sH)$ w.r.t. the prime ideal $\G p_\la \su \hoc$. This is a standard result about localizations, see for example \cite[Chapter IV, \S 2]{Ser:LAM}.

We thus only need to prove that $H(g(A),\sH)_\la \cong H(g(A),\sH)(\la)$.

Let $f$ be a function in the multiplicative subset $\hoc\setminus \G p_\la$. Since $f(\la) \neq 0$ we have that the image of $f$ in $\C O_\la$ is invertible. It therefore follows from Proposition \ref{p:powfac} that the induced homomorphism $f : H(g(A),\sH)(\la) \to H(g(A),\sH)(\la)$ is invertible.

Let $E_\la : H(g(A),\sH) \to H(g(A),\sH)$ denote the projection onto $H(g(A),\sH)(\la)$ relative to the decomposition $\op_{\la \in Z(g)} H(g(A),\sH)(\la) \cong H(g(A),\sH)$.

We then have a well-defined homomorphism of graded modules $H(g(A),\sH)_\la \to H(f(A),\sH)(\la)$ defined by $\xi/f \mapsto f^{-1}(E_\la \xi)$. We claim that this homomorphism is an isomorphism with inverse given by $\xi \mapsto \xi/1$. To prove this claim we mainly need to show that $(E_\la \xi)/1 = \xi/1$ in $H(g(A),\sH)_\la$. Or in other words, we need to prove that $\xi/1 = 0$ whenever $\xi \in H(g(A),\sH)(\mu)$ for some $\mu \neq \la$. However, for each such $\xi$ we can find a polynomial $p \in \hoc\setminus \G p_\la$ such that $p \cd \xi = 0$. Indeed, we could choose $p = (z_i - \mu_i)^k$ where $\mu_i \neq \la_i$ and $k \in \nn$ is large. But this implies that $\xi/1 = 0$ as desired.
\end{proof}

\section{Regularity}\label{s:reg} In this section the following general assumption will be in effect.
\bigskip

\noindent \emph{
Let $g : \T{Sp}(A) \to \cc^m$ be a holomorphic map on the Taylor spectrum of our commuting tuple $A = (A_1,\ldots,A_n)$ with $m \geq n$. Suppose that $g(A)$ is Fredholm and that $\la \in Z(g)$ is a common zero which is also a regular point for the first $n$ coordinates $(g_1,\ldots,g_n) : \T{Sp}(A) \to \cc^n$.}
\bigskip

%
%
%

Let $p_n(g) := (g_1,\ldots,g_n)$ and recall that $\la$ is said to be a regular point for $p_n(g)$ when the determinant of the Jacobian matrix $D(g_1,\ldots,g_n) = \T{det}\big(\frac{\pa g_i}{\pa z_j} \big)$ is non-trivial at $\la$. Most importantly for our purposes this implies that the homomorphism $\C O_0 \to \C O_\la$, $f \mapsto f \ci p_n(g)$ is an isomorphism of local rings. In particular we get the identity $\G m_\la = g_1 \C O_\la \plp g_n \C O_\la$ where $\G m_\la$ denotes the maximal ideal in $\C O_\la$. See \cite[Theorem 1.19]{Kod:CDC}.
%

As in Section \ref{s:nilban}, the notation $\sH_\la$ refers to the localization of the module $\sH$ w.r.t. the prime ideal $\G p_\la \su \hoc$ of holomorphic functions which vanish at $\la$.

In the next lemmas we prove various triviality results for actions on Koszul homology groups.

\begin{lemma}
Let $j \in \{1,\ldots,m\}$ and let $B = (B_1,\ldots,B_k)$ be a commuting tuple such that $A \op B$ commutes as well. Let $f \in \hoc$ and suppose that $f(\mu) = 0$ for some $\mu \in \T{Sp}(A)$. Then the endomorphism $H(f(A)) : H\big((A - \mu) \op B,\sH\big) \to H\big((A - \mu) \op B,\sH\big)$ is trivial.
\end{lemma}
\begin{proof}
Let $p_n : \cc^{n+k} \to \cc^n$ denote the projection onto the first $n$ coordinates. Then $f(A) = (f \ci p_n)(A \op B)$. See f. ex. \cite[Theorem 5.2.3]{EsPu:SDA}.

It now follows from \cite[Proposition 4.6]{Tay:ACS} that the induced endomorphism
\[
H(f(A)) = H\big( (f \ci p_n)(A \op B)\big)
: H\big((A - \mu) \op B,\sH\big) \to  H\big((A - \mu) \op B,\sH\big)
\]
is given by the scalar multiplication with $(f \ci p_n)(\mu,0) = 0$. This proves the lemma.
\end{proof}

\begin{lemma}
The endomorphism 
\[
H(z_j - \la_j) : H\big(g \op p_{j-1}(z-\la),\sH_\la\big) \to 
H\big(g \op p_{j-1}(z-\la),\sH_\la\big)
\]
is trivial for each $j \in \{1,\ldots,n\}$ where $p_{j-1}(z- \la) := (z_1 - \la_1,\ldots, z_{j-1} - \la_{j-1})$.
\end{lemma}
\begin{proof}
By Proposition \ref{p:powfac} and Proposition \ref{p:locspe} the action of $\C O(\T{Sp}(A))$ on $H\big(g \op p_{j-1}(z - \la),\sH_\la\big)$ factorizes through the stalk $\C O_\la$. It now follows from the regularity of $p_n(g)$ that the maximal ideal $\G m_\la$ is given by $g_1 \C O_\la \plp g_n \C O_\la$. Since $z_j-\la_j$ determines an element in $\G m_\la$ it is enough to show that the endomorphism $H(g_i) : H\big(g\op p_{j-1}(z-\la), \sH_\la \big) \to H\big(g\op p_{j-1}(z-\la), \sH_\la \big)$ is trivial for each $i \in \{1,\ldots,n\}$. But this is a well-known property of the Koszul homology groups, see \cite[Chapter IV, Proposition 4]{Ser:LAM}.
\end{proof}

The following lemma allows us to apply our two triviality results to compute the dimensions of certain Koszul homology groups.

\begin{lemma}
Let $A = (A_1,\ldots,A_n)$ and $B = (B_1,\ldots,B_k)$ be commuting tuples on a vector space $V$ such that the union $A \op B$ commutes as well. Suppose that the induced homomorphism 
\[
H(B_j) : H(A \op (B_1,\ldots,B_{j-1}),V) \to H(A \op (B_1,\ldots,B_{j-1}),V)
\]
vanishes for each $j \in \{1,\ldots,k\}$. We then have the identity
\[
\T{Dim}_{\cc}H_q(A \op B,V) = \sum_{p = 0}^k {k \choose p} \T{Dim}_{\cc} H_{q-p}(A,V)
\]
of dimensions over $\cc$ for all $q \in \{0,\ldots,n+k\}$.
\end{lemma}
\begin{proof}
The proof runs by induction on the number of elements in $B$. The statement is obviously true when there are no elements in $B$. Thus, suppose that it is valid for some $k_0 \in \nn \cup \{0\}$. By Lemma \ref{l:koscon} we get the existence of a long exact sequence
\[
\begin{CD}
\ldots @<<< H_{q-1}(A \op p_{k_0}(B)) @<<< H_q(A \op p_{k_0+1}(B)) @<<<
H_q(A \op p_{k_0}(B)) \\
& & & & & & @AA{H_q(B_{k_0+1})}A \\
& & \ldots @>>> H_{q+1}(A \op p_{k_0+1}(B))
@>>> H_q(A \op p_{k_0}(B))
\end{CD}
\]
Now, since $H(B_{k_0+1}) : H(A \op (B_1,\ldots,B_{k_0})) \to H(A \op (B_1,\ldots,B_{k_0}))$ is trivial we can conclude that
\[
\T{Dim}_{\cc} H_q(A \op p_{k_0+1}(B),V) = \T{Dim}_{\cc} H_q(A \op p_{k_0}(B),V) + \T{Dim}_{\cc} H_{q-1}(A \op p_{k_0}(B),V)
\]
for all $q \in \{0,\ldots,n+k_0 +1\}$. The induction hypothesis then implies that
\[
\begin{split}
\T{Dim}_{\cc}H_q(A \op p_{k_0+1}(B),V)
& = \sum_{p=0}^{k_0+1} \left( {k_0 \choose p} + {k_0 \choose p-1}\right)\T{Dim}_{\cc} H_{q-p}(A,V) \\
& = \sum_{p=0}^{k_0 +1}{k_0 +1 \choose p} \T{Dim}H_{q - p}(A,V).
\end{split}
\]
This proves the lemma.
\end{proof}

The above results can be combined into:

\begin{lemma}\label{l:indbin}
For each $q \in \{0,\ldots,n+m\}$ the following equalities hold:
\begin{eqnarray}\label{eq:binomial}
\T{Dim}_{\cc}H_q\big((z - \la) \op g, \sH_\la \big) = \sum_{p = 0}^n {n \choose p} \T{Dim}_{\cc}H_{q-p}(g,\sH_\la)\\
\T{Dim}_{\cc}H_q\big((z - \la) \op g, \sH_\la \big) =\sum_{p=0}^m {m \choose p} \T{Dim}_{\cc} H_{q-p}( A-\la,\sH).
\end{eqnarray}
\end{lemma}
\begin{proof}
Remark the existence of a graded isomorphism $H(z-\la,\sH_\la) \cong H(A-\la,\sH)$. Indeed, we have that $H(A - \la,\sH)_\la \cong H(A -\la,\sH)$ since each holomorphic function $f \in \C O(\T{Sp}(A))$ acts by scalar multiplication with $f(\la) \in \cc$ on the homology group $H(A -\la,\sH)$.

The desired identities now follow immediately from the above lemmas. Notice that $H_*((z-\la) \op g,\sH_\la) \cong H_*(g \op (z-\la),\sH_\la)$ by the symmetry of the Koszul homology groups. See f. ex. \cite[Theorem 3.2.3]{Kaa:JTS}.
\end{proof}

Let $R(n) : \nn_0^{m+1} \to \nn_0^{n+m+1}$ and $R(m) : \nn_0^{n+1} \to \nn_0^{n+m+1}$ denote the linear maps represented by the matrices
\[
R_{ij}(n) := {n \choose i -j} \q 
R_{ij}(m) := {m \choose i -j}
\]
The map $R(n)$ is clearly injective with left inverse $L(n) : \nn_0^{n+m+1} \to \nn_0^{m+1}$ represented by the matrix
\[
L_{ij}(n) := (-1)^{i-j} {n + i-j -1 \choose i -j}. 
\]

\begin{lemma}\label{l:comid}
The composition $L(n) \ci R(m) : \nn_0^{n+1} \to \nn_0^{m+1}$ is represented by the matrix
\[
\big( L(n) \ci R(m) \big)_{ij} = {m-n \choose i-j}.
\]
\end{lemma}
\begin{proof}
This follows by an application of the combinatorial identity
\[
\sum_{k=0}^{i-j} (-1)^k{n + k-1 \choose k}{m \choose i-j-k}
= {m-n \choose i-j}.
\]
\end{proof}

We are now ready to prove the main result of this section.

\begin{prop}\label{p:dimhom}
Suppose that $g(A)$ is Fredholm and that $\la \in Z(g)$ is a regular point for $(g_1,\ldots,g_n) : \T{Sp}(A) \to \cc^n$. We then have the equality
\[
\T{Dim}_{\cc} H_q(g,\sH_\la) = \sum_{p=0}^{m-n}
{m-n \choose p} \T{Dim}_{\cc} H_{q-p}(A-\la,\sH)
\]
for each $q \in \{0,\ldots,m\}$.
\end{prop}
\begin{proof}
Let $\eta \in \nn_0^{n+m+1}$, $\xi \in \nn_0^{m+1}$ and $\ze \in \nn_0^{n+1}$ be defined by
\[ 
\begin{split}
\eta_p & := \T{Dim}_\cc H_p((z-\la) \op g,\sH_\la) \q  p = 0 ,\ldots,n+m \\
\xi_q & := \T{Dim}_\cc H_q(g,\sH_\la) \qqq q = 0,\ldots,m \q \, \, \T{ and } \\
\ze_k & := \T{Dim}_\cc H_k(A - \la,\sH) \qq  k = 0,\ldots,n.
\end{split}
\]
It then follows from Lemma \ref{l:indbin} that
\[
R(n)(\xi) = \eta = R(m)(\ze).
\]
In particular we get that $\xi = L(n) R(m)(\ze)$ and the proposition is proved by using Lemma \ref{l:comid}.
\end{proof}

As an application of Proposition \ref{p:dimhom} we obtain the local index theorem in the regular case.

\begin{theorem}\label{t:regind}
Suppose that $g(A)$ is Fredholm and that $\la \in Z(g)$ is a regular point for $(g_1,\ldots,g_n) : \T{Sp}(A) \to \cc^n$. Then
\[
\T{Ind}_\la(g(A)) = \fork{ccc}{
\T{Ind}(A - \la) & \T{for} & m = n \\
0 & \T{for} & m > n
}
\]
\end{theorem}
\begin{proof}
The result follows immediately from Proposition \ref{p:dimhom} when $m=n$ since $\T{Dim}_\cc H_q(g,\sH_\la) = \T{Dim}_\cc H_q(A-\la,\sH)$ for all $q \in \{0,\ldots,n\}$ in this case. Notice that $H_q(g,\sH_\la) \cong H_q(g(A),\sH)(\la)$ by Proposition \ref{p:locspe}.

Thus suppose that $m > n$. By Proposition \ref{p:dimhom} we have that
\[
\begin{split}
\T{Ind}_\la(g(A))
& = \sum_{q = 0}^m (-1)^{q+1} \T{Dim}_\cc H_q(g,\sH_\la) \\
& = \sum_{q = 0}^m (-1)^{q+1} \sum_{p = 0}^{m-n} {m-n \choose p}
\T{Dim}_\cc H_{q-p}(A-\la,\sH) \\
& = \sum_{j=0}^n (-1)^{j+1} \sum_{p=0}^{m-n} (-1)^p {m-n \choose p} \T{Dim}_\cc H_j(A-\la,\sH) \\
& = 0.
\end{split}
\]
This ends the proof of the theorem.
\end{proof}

\section{The global index theorem}\label{s:gloind}
Throughout this section the following condition will be in effect:
\bigskip

\noindent \emph{Let $g : \T{Sp}(A) \to \cc^m$ be a holomorphic map on the Taylor spectrum of the commuting tuple $A=(A_1,\ldots,A_n)$ of bounded operators such that $g(A)$ is Fredholm. }
\bigskip

As in the last section we let $p_n(g) := (g_1,\ldots,g_n)$.

\begin{notation}
The notation $C_{p_n(g)} \su \T{Sp}(A)$ refers to the set of critical points for $p_n(g) : \T{Sp}(A) \to \cc^n$, thus the points for which the determinant of the Jacobian matrix $\T{det}\big(\frac{\pa g_i}{\pa z_j} \big)$ vanishes.
\end{notation}

It follows by Sard's Lemma that the image $p_n(g)( C_{p_n(g)}) \su \cc^n$ has Lebesgue measure zero. See \cite[Theorem 3.1]{Sch:NFA}.

\begin{theorem}\label{t:inthetri}
Suppose that $g(A)$ is Fredholm and that $m > n$. Then $\T{Ind}(g(A))=0$.
\end{theorem}
\begin{proof}
By the homotopy invariance of the index, \cite[Theorem 3]{Cur:FOD}, we can find an open neighborhood $U$ of $0 \in \cc^m$ such that $g(A) - z = (g - z)(A)$ is Fredholm with $\T{Ind}(g(A) - z) = \T{Ind}(g(A))$ for all $z \in U$. Since $p_n(g)(C_{p_n(g)})$ has measure zero, we may thus assume, without loss of generality, that $0 \notin p_n(g)(C_{p_n(g)})$. This clearly implies that every $\la \in Z(g)$ is a regular point for $p_n(g) : \T{Sp}(A) \to \cc^n$. By Theorem \ref{t:regind} we then get that $\T{Ind}_\la(g(A)) = 0$ for all $\la \in Z(g)$. The result of the theorem now follows by an application of Proposition \ref{p:sumind}.
\end{proof}

\begin{notation}
Let $\Om := \T{Int}(\T{Sp}(A))$ denote the interior of the Taylor spectrum.
\end{notation}

Suppose now that $m = n$ and that $\la \in Z(g) \cap \Om$. By Theorem \ref{t:equzerspe} we have that $Z(g)$ is finite and we can thus find a ball $B_{\ep} \su \Om$ such that $Z(g) \cap B_\ep = \{\la\}$. 

\begin{notation}
The notation $\T{deg}_{\la}(g) := \T{deg}(0,g,B_\ep) \in \nn$ refers to the degree of the holomorphic map $g : B_{\ep} \to \rr^{2n}$ at the value $0$ (where $\cc^n$ is identified with $\rr^{2n}$). This strictly positive integer will be called the \emph{local degree} of $g$ at $\la$.
\end{notation}

For later use we record the following result which can be deduced from \cite[Proposition 2.4]{EiLe:ADG}.

\begin{lemma}\label{l:locdeg}
The local degree $\T{deg}_{\la}(g)$ coincides with the vector space dimension of the quotient $\C O_\la/g\C O_\la := \C O_\la/(g_1\C O_\la \plp g_n \C O_\la)$ for each $\la \in Z(g) \cap \Om$.
\end{lemma}

The global index theorem can now be proved in the main case of interest. See also \cite[Theorem 10.3.13]{EsPu:SDA}.

\begin{theorem}\label{t:indthe}
Suppose that $g(A)$ is Fredholm and that $m = n$. Then
\[
\T{Ind}(g(A)) = \sum_{\la \in Z(g) \cap \Om} \T{deg}_{\la}(g) \cd \T{Ind}(A - \la).
\]
\end{theorem}
\begin{proof}
Let us choose a small ball $B_{\ep(\la)} \su \cc^n$ for each $\la \in Z(g)$ such that the following conditions hold:
\begin{enumerate}
\item $\la \in B_{\ep(\la)}$.
\item $B_{\ep(\la)} \cap B_{\ep(\mu)} = \emptyset$ whenever $\la \neq \mu$.
\item $B_{\ep(\la)} \cap \T{Sp}_{\T{ess}}(A) = \emptyset$ for all $\la \in Z(g)$.
\item $B_{\ep(\la)} \su \Om$ for all $\la \in Z(g) \cap \Om$.
\end{enumerate}

Recall that the image $g(C_g)$ of the set of critical points for $g$ has measure zero. We can thus find a sequence $\{\al_k\}$ of elements in $\cc^n$ which converges to zero such that the set of common zeroes $Z(g - \al_k)$ consists entirely of regular points for all $k \in \nn$.

Let us fix some element $\al := \al_k$ from the above sequence.

By the homotopy invariance of the index and of the degree we may assume that
\begin{equation}\label{eq:deg}
\T{Ind}\big((g - \al)(A)\big) = \T{Ind}(g(A)) \q \T{and} \q
\T{deg}(0,g-\al,B_{\ep(\la)}) = \T{deg}_\la(g)
\end{equation}
for all $\la \in Z(g) \cap \Om$. See \cite[Theorem 3.16]{Sch:NFA} and \cite[Theorem 3]{Cur:FOD}.

Furthermore, by a compactness argument we may assume that
\[
Z(g - \al) = Z(g-\al) \cap B \,\, \, \, \T{where } \, \, \, \, B := \cup_{\la \in Z(g)} B_{\ep(\la)}.
\]

An application of Theorem \ref{t:regind} now yields the second of the following identities:
\[
\begin{split}
\T{Ind}(g(A))
& = \T{Ind}( (g-\al)(A))
= \sum_{\mu \in Z(g-\al)} \T{Ind}(A - \mu) \\
& = \sum_{\la \in Z(g)} \, \sum_{\mu \in Z(g-\al) \cap B_{\ep(\la)}} \T{Ind}(A - \mu).
\end{split}
\]

Suppose that $\la \in Z(g) \cap \Om$. Since $Z(g-\al) \cap C_g = \emptyset$, the degree $\T{deg}(0,g-\al,B_{\ep(\la)}) \in \nn$ is nothing but the number of zeroes of the function $g - \al : B_{\ep(\la)} \to \cc^n$, thus
\[
\T{deg}_\la(g) = \T{deg}(0,g-\al,B_{\ep(\la)}) = \big| Z(g-\al) \cap B_{\ep(\la)} \big|.
\]

Furthermore, since $B_{\ep(\la)} \cap \T{Sp}_{\T{ess}}(A) = \emptyset$ we have that $\T{Ind}(A-\mu) = \T{Ind}(A - \la)$ for all $\mu \in B_{\ep(\la)}$. This allows us to compute as follows:
\[
\begin{split}
\sum_{\mu \in Z(g-\al) \cap B_{\ep(\la)}} \T{Ind}(A - \mu)
= \big| Z(g-\al) \cap B_{\ep(\la)}\big| \cd \T{Ind}(A - \la)
= \T{deg}_\la(g) \cd \T{Ind}(A - \la).
\end{split}
\]

To finish the proof of the global index theorem we therefore only need to verify the identity
\[
\sum_{\mu \in Z(g- \al) \cap B_{\ep}(\la)} \T{Ind}(A - \mu) = 0
\]
whenever $\la \in Z(g) \cap \pa\T{Sp}(A)$. But for such a $\la$, $\T{Ind}(A-\la) = 0$, and the desired identity follows since $\T{Ind}(A-\la) = \T{Ind}(A - \mu)$ for all $\mu \in B_{\ep}(\la)$.
\end{proof}

The next result will only rely on the finiteness of the zero set $Z(g)$ and some considerations on dimensions.

\begin{theorem}
Suppose that $g(A)$ is Fredholm and that $m < n$. Then one of the following holds: Either $Z(g) \su \T{Sp}(A) \setminus \T{cl}(\Om)$ or $g(A)$ is invertible.
\end{theorem}
\begin{proof}
Notice that $g(A)$ is invertible if and only if $Z(g) = \emptyset$. Thus suppose for contradiction that $Z(g) \cap \T{cl}(\Om) \neq \emptyset$.

Let $\la \in Z(g) \cap \T{cl}(\Om)$. Suppose first that $\la \in \Omega$. From Theorem \ref{t:equzerspe} we know that $Z(g)$ is a finite set. In particular we get that $\la$ is an isolated point. But this is a contradiction since the analytic dimension of $\cc^n$ at $\la$ is $n$ whereas the number of coordinates for $g = (g_1,\ldots,g_m)$ is strictly less than $n$. See \cite[Chapter 5, \S 3.1]{GrRe:CAS}.

Suppose next that $\la \in \pa \Omega$. We can then find a sequence $\{\la_k\}$ of elements in $\Omega$ which converges to $\la$. Since $g(\la) = 0$ and $g : \T{Sp}(A) \to \cc^m$ is continuous we get that $\{g(\la_k)\}$ converges to zero. Furthermore, since $g(A)$ is Fredholm we have that $0 \notin g\big( \T{Sp}_{\T{ess}}(A) \big) = \T{Sp}_{\T{ess}}(g(A))$. By the compactness of $\T{Sp}_{\T{ess}}(A) \su \cc^n$ we can thus find an open neighborhood $U$ of zero such that $U \cap \T{Sp}_{\T{ess}}(g(A)) = \emptyset$. Let us choose a $k \in \nn$ such that $\mu := g(\la_k) \in U$.

The function $h := g - \mu : \T{Sp}(A) \to \cc^m$ is now holomorphic and the set $Z(h) \cap \Omega$ is non-trivial. Furthermore, since $0 \notin h(\T{Sp}_{\T{ess}}(A)) = \T{Sp}_{\T{ess}}(g(A)) - \mu$ we conclude that $h(A)$ is Fredholm. The above argument then leads to a contradiction in this case as well.
\end{proof}

\begin{remark}
It is in general \emph{not} true that $\T{Ind}(g(A)) = 0$ when $m < n$. As an example, suppose that $n = 2$ and that $A = (B,B)$ where $B : \sH \to \sH$ is a Fredholm operator with $\T{Ind}(B) \neq 0$. If $g : \cc^2 \to \cc$ is the projection onto the first coordinate we get that $\T{Ind}(g(A)) = \T{Ind}(B)$.
\end{remark}

\section{Reciprocity of local indices}\label{s:respro}

\noindent In this section the following notational convention will be used:
\bigskip

\noindent \emph{Let $B = (B_1,\ldots,B_n)$ be an extra commuting tuple of the same length as $A = (A_1,\ldots,A_n)$ but acting on a possibly different Hilbert space $\sG$. The letter $\sK := \sH \hot \sG$ will refer to the completed tensor product of our two Hilbert spaces.}
\bigskip

\noindent The purpose of this section is twofold. We will investigate an analytic version of the algebraic localization procedure introduced in Section \ref{s:nilban}. Furthermore, we will prove a reciprocity formula which relates the local indices with respect to the two independent commuting tuples $A$ and $B$.

Before presenting these results, we need some preliminary lemmas. They are mainly concerned with the behaviour of indices and spectra under the tensor product operation.

\begin{lemma}\label{l:tenfun}
There is an inclusion
\[
\T{Sp}\big( (A \ot 1) \op (1 \ot B) \big) \su \T{Sp}(A) \ti \T{Sp}(B)
\]
of spectra, where $A \ot 1$ and $1 \ot B$ act on $\sK := \sH \hot \sG$. Furthermore, for each $f \in \C O(\T{Sp}(A))$ and $h \in \C O(\T{Sp}(B))$, we have that
\begin{equation}\label{eq:tenfun}
f(A \ot 1) = f(A) \ot 1 \q \T{and} \q h(1 \ot B) = 1 \ot h(B).
\end{equation}
\end{lemma}
\begin{proof}
Let $(\la,\mu) \notin \T{Sp}(A) \ti \T{Sp}(B)$ and let us show that $(\la,\mu) \notin \T{Sp}\big( (A \ot 1) \op (1 \ot B) \big)$. Without loss of generality we may assume that $\la \notin \T{Sp}(A)$. It is then sufficient to prove that the Koszul complex $K(A \ot 1 - \la, \sH \hot \sG)$ is exact. But this is true since the Koszul complex $K(A - \la,\sH)$ is exact and since the functor $\cd \, \hot \sG$ preserves exact sequences of Hilbert spaces.

Let us show that $f(A \ot 1) = f(A) \ot 1$. The other identity in \eqref{eq:tenfun} follows by a similar argument. Let $\ga \in \sL(\sG,\cc)$ be a linear functional on $\sG$. It is then enough to prove the identity
\[
f(A)(1 \ot \ga) = (1 \ot \ga)f(A \ot 1).
\]
But this follows from Proposition \ref{p:functoriality} since $A_i (1 \ot \ga) = (1 \ot \ga)(A_i \ot 1)$ for all $i \in \{1,\ldots,n\}$. Notice that $\T{Sp}(A) \su \T{Sp}(A \ot 1)$ by an application of the first part of the lemma.
\end{proof}

\begin{lemma}\label{l:isodif}
Let $C$ be an extra commuting tuple on $\sK$ such that $(A \ot 1 - 1 \ot B) \op C$ commutes. Let $f : \T{Sp}(A) \cup \T{Sp}(B) \to \cc^m$ be holomorphic. Then there exists a graded isomorphism of Koszul homology groups
\[
H\big( (A \ot 1  - 1 \ot B) \op C \op (f(A) \ot 1), \sK \big)
\cong H\big( (A \ot 1 - 1 \ot B) \op C \op (1 \ot f(B)), \sK \big).
\]
\end{lemma}
\begin{proof}
We use the short notation $f(A) := f(A) \ot 1$ and $f(B) := 1 \ot f(B)$ as well as $D := (A \ot 1 - 1 \ot B) \op C$.

Without loss of generality we may assume that $m = 1$. By Lemma \ref{l:koscon} we have two long exact sequences of homology groups
\[
\xymatrix{
\ldots \ar[r] 
& H_{*+1}( D \op f(A),\sK) \ar[r]
& H_*(D,\sK) \ar[r]^-{H_*(f(A))}
& H_*(D,\sK) \ar[d] \\
& \hspace{30pt} \ldots \hspace{30pt}
& H_{*-1}(D,\sK) \ar[l]
& H_*(D \op f(A),\sK) \ar[l] }
\]
\[
\xymatrix{
\ldots \ar[r] 
& H_{*+1}( D \op f(B),\sK) \ar[r]
& H_*(D,\sK) \ar[r]^-{H_*(f(B))}
& H_*(D,\sK) \ar[d] \\
& \hspace{30pt} \ldots \hspace{30pt} 
& H_{*-1}(D,\sK) \ar[l]
& H_*(D \op f(B),\sK) \ar[l] }
\]
However, by an application of \cite[Proposition 2.5.9]{EsPu:SDA} and Lemma \ref{l:tenfun} we get that the linear maps $H_*(f(A) \ot 1)$ and $H_*(1 \ot f(B))$ are identical. This implies the result of the lemma since we are working in the category of vector spaces over $\cc$. Notice however that the isomorphism which we obtain is not canonical since it depends on the choice of complementary subspaces.
\end{proof}

\begin{lemma}\label{l:indproind}
Suppose that $A$ and $B$ are Fredholm. Then the commuting tuple $(A \ot 1) \op (1 \ot B)$ is Fredholm on $\sH \hot \sG$ with index
\[
\T{Ind}\big( (A \ot 1) \op (1 \ot B) \big) = -\T{Ind}(A) \cd \T{Ind}(B).
\]
\end{lemma}
\begin{proof}
By Proposition \ref{p:speind} it suffices to show that $H_p\big(A \ot 1, H_q((1 \ot B),\sK)\big)$ is finite dimensional for each $p,q \in \{0,\ldots,n\}$ and that
\[
\sum_{p,q} (-1)^{p+q+1} \T{Dim}\Big( H_p\big(A \ot 1, H_q((1 \ot B),\sK)\big) \Big) = \T{Ind}(A) \cd \T{Ind}(B).
\]

Using the Fredholmness of $A$ and $B$ together with the exactness of the Hilbert space tensor product we get the isomorphisms
\[
H_p\big(A \ot 1, H_q((1 \ot B),\sK)\big) \cong H_p\big(A \ot 1, \sH \ot H_q(B,\sG) \big) \cong H_p(A,\sH) \ot H_q(B,\sG).
\]
See also the proof of Proposition \ref{p:locana}.

This implies that
\[
\begin{split}
& \sum_{p,q} (-1)^{p+q+1} \T{Dim}\Big( H_p\big(A \ot 1, H_q((1 \ot B),\sK)\big) \Big) \\
& \q = \sum_p (-1)^p \T{Dim}\big( H_p(A,\sH) \big) \cd \sum_q (-1)^{q+1} \T{Dim} \big( H_q(B,\sG) \big)
= -\T{Ind}(A) \cd \T{Ind}(B)
\end{split}
\]
and the lemma is proved.

\end{proof}

\begin{lemma}\label{l:niliso}
Let $V$ be a vector space of finite dimension over $\cc$. Suppose that $A = (A_1,\ldots,A_n)$ is Fredholm on $\sH$ and let $C = (C_1,\ldots,C_n)$ be a commuting tuple on $V$ with $C_i$ nilpotent for all $i \in \{1,\ldots,n\}$. We then have the identity
\[
\T{Ind}(A \ot 1 + 1 \ot C) = \T{Ind}(A) \cd \T{Dim}(V)
\]
of Fredholm indices where the commuting tuple $A \ot 1 + 1 \ot C$ acts on the Hilbert space $\sH \ot V$.
\end{lemma}
\begin{proof}
We argue by induction on the dimension of $V$. 

Suppose that $\T{Dim}(V) = 1$. This implies that $C$ is trivial. Let $\xi \in V$ be a non-trivial vector. The isomorphism of Hilbert space $\sH \ot V \cong \sH$ given by $\eta \ot \al \cd \xi \mapsto \al \cd \eta$ then induces an isomorphism of Koszul homology groups $H_*(A \ot 1, \sH \ot V) \cong H_*
(A,\sH)$. This proves the statement in this case.

Suppose that the statement is true for $\T{Dim}(V) = k$ for some $k \geq 1$ and suppose that $\T{Dim}(V) = k+1$. Since $C$ is nilpotent on $V$ there exists a non-trivial vector $\xi \in V$ with $C_i(\xi) = 0$ for all $i \in \{1,\ldots,n\}$. See Theorem \ref{t:spedec}. Let $W := \cc \cd \xi \su V$ denote the span of $\xi$ in $V$. There is a long exact sequence of homology groups 
\[
\begin{CD}
\ldots @>>> H_{*+1}(D,\sH \ot V/W) @>>> H_*(D,\sH \ot W) @>>>
H_*(D,\sH \ot V) \\
& & & & & & @VVV \\
& &  \ldots @<<< H_{*-1}(D,\sH \ot W) @<<< H_*(D,\sH \ot V/W)
\end{CD}
\]
where $D := A \ot 1 + 1 \ot C$. By the induction hypothesis and the induction start we get that $\T{Ind}(D|_{\sH \ot W}) = \T{Ind}(A)$ and $\T{Ind}(D|_{\sH \ot V/W}) = \T{Dim}(V/W) \cd \T{Ind}(A)$. This proves the induction step by the additivity of the Euler characteristic.
\end{proof}

%
%
%
%
%

The following result shows that the passage from $\sH$ to the tensor product $\sH \hot \sG$ can be used as an analytic counterpart of the algebraic localization at the zeroes of a holomorphic function as far as index computations are concerned. This will turn out to be important for the proof of our local index theorem.

%
\begin{prop}\label{p:locana}
Let $g : \T{Sp}(A) \to \cc^m$ be holomorphic. Suppose that $g(A)$ is Fredholm and that $\la - B$ is Fredholm for all $\la \in Z(g)$. Then the commuting tuple $(A \ot 1 - 1 \ot B) \op (g(A) \ot 1)$ is Fredholm on $\sH \hot \sG$ and the index is given by
\[
\T{Ind}\big( (A \ot 1 - 1 \ot B) \op (g(A) \ot 1)\big) = 
-\sum_{\la \in Z(g)} \T{Ind}_\la(g(A)) \cd \T{Ind}(\la - B).
\]
\end{prop}
\begin{proof}
By Proposition \ref{p:speind} it is enough to show that $H_p\big(A \ot 1 - 1 \ot B, H_q(g(A) \ot 1,\sK  \big)$ is finite dimensional for all $p \in \{0,\ldots,n\}$ and $q \in \{0,\ldots,m\}$ and that
\begin{equation}\label{eq:etwoind}
\begin{split}
& \sum_{p,q} (-1)^{p+q+1} \T{Dim}_{\cc}H_p\big( A \ot 1 - 1 \ot B, H_q( g(A) \ot 1,\sK) \big) \\
& \q = -\sum_{\la \in Z(g)} \T{Ind}_\la(g(A)) \cd \T{Ind}(\la - B).
\end{split}
\end{equation}

Let us fix a $q \in \{0,\ldots,m\}$. Note that $H_q(g(A) \ot 1,\sK) \cong H_q(g(A),\sH) \ot \sG$. To see this, it suffices to recall that the functor $\cd \, \hot \sG$ sends exact sequences of Hilbert spaces to exact sequences of Hilbert spaces. Furthermore, the image of the differential of the Koszul complex $K(g(A),\sH)$ is closed by the Fredholmness assumption on $g(A)$.

Using the Fredholmness of $g(A)$ one more time, we can decompose $H_q(g(A),\sH)$ into the generalized eigenspaces
\[
H_q(g(A),\sH) \cong \bop_{\la \in Z(g)} H_q(g(A),\sH)(\la).
\]
for the commuting tuple $H_q(A)$.

Let $\la \in Z(g)$. Since $\la - B$ is Fredholm on $\sG$ and $H_q(A)(\la) - \la$ is nilpotent on $H_q(g(A),\sH)(\la)$ we get from Lemma \ref{l:niliso} that
\[
\begin{split}
\T{Ind}\big( H_q(A)(\la) \ot 1 - 1 \ot B \big) 
& = \T{Ind}\Big( \big(H_q(A)(\la) - \la \big) \ot 1 + 1 \ot (\la - B) \Big) \\
& = \T{Dim}_{\cc}\big( H_q(g(A),\sH)(\la) \big) \cd \T{Ind}(\la - B)
\end{split}
\] 
where $H_q(A)(\la) \ot 1 - 1 \ot B$ acts on the Hilbert space $H_q(g(A),\sH)(\la) \ot \sG$.

Combining the above achievements we thus get that
\[
\begin{split}
& \sum_q (-1)^q \cd \T{Ind}\big( H_q(A) \ot 1 - 1 \ot B\big) \\
& \q = \sum_{\la\in Z(g)} \sum_q (-1)^q \T{Ind}\big( H_q(A)(\la) \ot 1 - 1 \ot  B \big)  \\
& \q = \sum_{\la \in Z(g)}\sum_q (-1)^q \T{Dim}_{\cc} \big( H_q(g(A),\sH)(\la) \big) \cd \T{Ind}(\la - B) \\
& = -\sum_{\la \in Z(g)} \T{Ind}_\la(g(A)) \cd \T{Ind}(\la -B)
\end{split}
\]
where $H_q(A) \ot 1 - 1 \ot B$ acts on the Hilbert space $H_q(g(A),\sH) \ot \sG$. But this is equivalent to the identity in \eqref{eq:etwoind} and the proposition is proved.
\end{proof}

We are now in position to prove the reciprocity formula for local indices.

\begin{theorem}\label{t:respro}
Let $g : \T{Sp}(A) \cup \T{Sp}(B) \to \cc^m$ be holomorphic. Suppose that $g(A)$ and $g(B)$ are Fredholm and that the sets $Z(g) \cap \T{Sp}(A) \cap \T{Sp}_{\T{ess}}(B)$ and $Z(g) \cap \T{Sp}(B) \cap \T{Sp}_{\T{ess}}(A)$ are empty.
We then have the identity
\[
\sum_{\mu \in Z(g) \cap \T{Sp}(B)}\T{Ind}(\mu -A) \cd \T{Ind}_\mu(g(B))
= \sum_{\la \in Z(g) \cap \T{Sp}(A)} \T{Ind}_\la(g(A)) \cd \T{Ind}(\la - B).
\]
\end{theorem}
\begin{proof}
By an application of Proposition \ref{p:locana} we get the identities
\[
\begin{split}
& \T{Ind}\big( (A \ot 1 - 1 \ot B) \op (g(A) \ot 1)\big) = 
-\sum_{\la \in Z(g) \cap \T{Sp}(A)} \T{Ind}_\la(g(A)) \cd \T{Ind}(\la - B)
\, \, \, \,  \T{and} \\
& \T{Ind}\big( (A \ot 1 - 1 \ot B) \op (1 \ot g(B))\big) = 
-\sum_{\mu \in Z(g) \cap \T{Sp}(B)} \T{Ind}(\mu - A) \cd \T{Ind}_\mu(g(B)) 
\end{split}
\]
But this entails the result of the theorem since
\[
\T{Ind}\big( (A \ot 1 - 1 \ot B) \op (g(A) \ot 1)\big) = 
\T{Ind}\big( (A \ot 1 - 1 \ot B) \op (1 \ot g(B))\big)
\]
by an application of Lemma \ref{l:isodif}.
\end{proof}

\section{The local index theorem}\label{s:locind}
Throughout this section the following conditions will be in effect:
\bigskip

\noindent \emph{Let $g : \T{Sp}(A) \to \cc^m$ be a holomorphic map on the Taylor spectrum of the commuting tuple $A=(A_1,\ldots,A_n)$ of bounded operators such that $g(A)$ is Fredholm.}
\bigskip

\noindent \emph{For each $\ep > 0$ and each $\la \in \cc^n$, let $\sG_\ep(\la)$ be a Hilbert space and let $B = (B_1,\ldots,B_n)$ be a commuting tuple of bounded operators on $\sG_\ep(\la)$ such that the following conditions are satisfied:}
\emph{
\begin{enumerate}
\item The Taylor spectrum of $B$ is included in the closed ball $\T{cl}\big( \B B_{\ep}(\la)\big)$ in $\cc^n$ with radius $\ep > 0$ and center $\la \in \cc^n$.
\item The commuting tuple $B - \la$ is Fredholm with $\T{Ind}(B - \la) = -1$.
\end{enumerate}}

An example of a pair $\big(\sG_\ep(\la),B\big)$ with the above properties consists of the Bergman space $H^2(B_\ep(\mu))$ and the commuting tuple $T_z = (T_{z_1},\ldots,T_{z_n})$ of Toeplitz operators associated with the coordinate functions. See Theorem \ref{t:pseudo}.

\begin{lemma}\label{l:analocpoi}
Let $\la \in \cc^n$ and let $\ep > 0$. Suppose that $\T{cl}\big( \B B_\ep(\la) \big) \cap Z(g) \su \{\la\}$. Then the commuting tuple
$\big( (A \ot 1 - 1 \ot B) \op (g(A) \ot 1) \big)$ is Fredholm on $\sH \hot \sG_\ep(\la)$ with index
\[
\T{Ind}\big( (A \ot 1 - 1 \ot B) \op (g(A) \ot 1) \big) = \T{Ind}_\la(g(A)).
\]
\end{lemma}

Note that $\T{Ind}_\la(g(A)) = 0$ whenever $\la \notin Z(g)$.

\begin{proof}
It follows from Proposition \ref{p:locana} and the conditions on $\big(\sG_\ep(\la),B \big)$ that $(A \ot 1 - 1 \ot B) \op (g(A) \ot 1)$ is Fredholm on $\sH \hot \sG_\ep(\la)$ with index
\[
\T{Ind}\big( (A \ot 1 - 1 \ot B) \op (g(A) \ot 1) \big)
= - \sum_{\mu \in Z(g)} \T{Ind}_\mu(g(A)) \cd \T{Ind}(\mu - B).
\]
Since $\T{Ind}(\mu - B) = 0$ for all $\mu \notin \T{cl}\big( \B B_\ep(\la) \big)$ and $\T{Ind}(\la - B) = - 1$ it follows that
\[
- \sum_{\mu \in Z(g)} \T{Ind}_\mu(g(A)) \cd \T{Ind}(\mu - B)
= \fork{ccc}{
\T{Ind}_\la(g(A)) & \T{for} & \la \in Z(g) \\
0 & \T{for} & \la \notin Z(g).
}
\]
This proves the lemma.
\end{proof}

\begin{notation}
Let $h : \T{Sp}(A) \ti \cc^n \to \cc^{n+m}$ denote the holomorphic map defined by $h(z,w) = (z-w,g(z))$.
\end{notation}

\begin{remark}\label{r:commap}
Let $B = (B_1,\ldots,B_n)$ be an arbitrary commuting tuple on a Hilbert space $\sG$. It follows from Lemma \ref{l:tenfun} that $h\big( (A \ot 1) \op (1 \ot B)\big) = (A \ot 1 - 1 \ot B) \op (g(A) \ot 1)$.
\end{remark}

\begin{theorem}
Suppose that $m > n$. Then $\T{Ind}_\la(g(A)) = 0$ for all $\la \in \cc^n$.
\end{theorem}
\begin{proof}
Since $g(A)$ is Fredholm, the set $Z(g)$ is finite. Choose an $\ep > 0$ such that $\T{cl}\big(\B B_\ep(\la)\big) \cap Z(g) \su \{\la\}$. The preceding remark and Lemma \ref{l:analocpoi} imply that $h\big( (A \ot 1) \op (1 \ot B) \big)$ is Fredholm with 
\[
\T{Ind}\big( h\big( (A \ot 1) \op (1 \ot B) \big)\big) = \T{Ind}_\la(g(A)).
\]
But Theorem \ref{t:inthetri} implies that
\[
\T{Ind}\big( h\big( (A \ot 1) \op (1 \ot B) \big) \big) = 0
\]
since $m+n > 2n$. This proves the theorem.
\end{proof}

Recall that $\Om := \T{Int}(\T{Sp}(A))$ denotes the interior of the Taylor spectrum of $A$.

\begin{lemma}\label{l:degcom}
Suppose that $\la \in \Om \cap Z(g)$ and that $n=m$. Then $\T{deg}_\la(g) = \T{deg}_{(\la,\la)}(h)$.
\end{lemma}

Note that $h(z,w) = 0 \lrar \big( z = w \T{ and } g(z) = 0 \big)$, thus $Z(h) = \{(\mu,\mu) \, | \, \mu \in Z(g)\}$ is a finite set. It follows in particular that $\T{deg}_{(\la,\la)}(h)$ is well-defined.

\begin{proof}
Recall from Lemma \ref{l:locdeg} that $\T{deg}_{(\la,\la)}(h) = \C O_{(\la,\la)}/h\C O_{(\la,\la)}$ where $h \C O_{(\la,\la)}$ denotes the ideal generated by the coordinates of $h$.

The map $\al : \C O_{(\la,\la)}/(z-w) \C O_{(\la,\la)} \to \C O_\la$ given by $f \mapsto f \ci \De$, where $\De : \cc^n \to \cc^{2n}$ is the diagonal map $\De(z) := (z,z)$, is an isomorphism of unital algebras over $\cc$. The inverse is given by $f \mapsto f \ci p_n$, where $p_n : \cc^{2n} \to \cc^n$ is the projection onto the first $n$ factors.

It follows in particular that $\al$ induces an isomorphism between the unital $\cc$-algebras
\[
\C O_{(\la,\la)}/h\C O_{(\la,\la)} = \C O_{(\la,\la)}/\big( (z-w) \C O_{(\la,\la)} + (g \ci p_n) \C O_{(\la,\la)}\big) \,\,\,\, \T{and} \,\,\,\, \C O_\la/g\C O_\la. 
\]
This proves the lemma.


\end{proof}

The main result of this paper can now be stated and proved. It expresses the local index at a point in terms of the local degree of the symbol $g$ and the locally constant index function of the variables $A$.

\begin{theorem}\label{t:localind}
Suppose that $g(A)$ is Fredholm, that $n=m$, and that $\la \in Z(g)$. The local index at $\la$ is then given by
\[
\T{Ind}_\la(g(A)) = \fork{ccc}{
0 & \T{for} & \la \in \pa\big( \T{Sp}(A) \big) \\
\T{deg}_\la(g) \cd \T{Ind}(\la - A) & \T{for} & \la \in \T{Int}\big( \T{Sp}(A) \big).
}
\]
\end{theorem}
\begin{proof}
Choose an $\ep > 0$ such that $\T{cl}\big( \B B_\ep(\la)\big) \cap Z(g) = \{\la\}$. It follows by Lemma \ref{l:analocpoi} and Remark \ref{r:commap} that $h\big( (A \ot 1) \op (1 \ot B) \big)$ is Fredholm on $\sH \hot \sG_\ep(\la)$ with
\[
\T{Ind}\big( h( (A \ot 1) \op (1 \ot B))\big) = \T{Ind}_\la(g(A)).
\]

Let $C : = (A \ot 1) \op (1 \ot B)$ and let $W := \T{Int}\big(\T{Sp}(C)\big)$. The global index theorem \ref{t:indthe} then implies that
\[
\T{Ind}(h(C)) = \sum_{\mu \in W \cap Z(h)} \T{deg}_\mu(h) \cd \T{Ind}(\mu - C).
\]

By Lemma \ref{l:tenfun} the open set $W$ is included in $\Om \ti \B B_\ep(\la)$. Furthermore, the set of zeroes of $h$ is the diagonal $Z(h) = \{(\mu,\mu)\, | \, \mu \in Z(g)\}$. The intersection is therefore given by
\[
W \cap Z(h) = \fork{ccc}{
\{(\la,\la)\} & \T{for} & \la \in \Om \\
\emptyset & \T{for} & \la \in \pa\big( \T{Sp}(A) \big).
} 
\]
This implies that
\[
\T{Ind}(h(C)) = \fork{ccc}{
0 & \T{for} & \la \in \pa\big( \T{Sp}(A) \big) \\
\T{deg}_{(\la,\la)}(h) \cd \T{Ind}\big( (\la,\la) - C\big)
& \T{for} & \la \in \Om.  
}
\]
But it follows from Lemma \ref{l:degcom} and Lemma \ref{l:indproind} that
\[
\begin{split}
& \T{deg}_{(\la,\la)}(h) \cd \T{Ind}((\la,\la) - C)
 = \T{deg}_\la(g) \cd \T{Ind}\big( ( (\la - A) \ot 1) \op ( 1 \ot (\la -B)) \big) \\
& \q = -   \T{deg}_\la(g) \cd \T{Ind}(\la- A) \cd \T{Ind}(\la - B)
= \T{deg}_\la(g) \cd \T{Ind}(\la- A)
\end{split}
\]
when $\la \in \Om$. This proves the theorem.
\end{proof}

%

\section{Examples}\label{s:exa}
\subsection{Bergman space}
Let $\Omega \su \cc^n$ be a bounded open set. We let $L^2(\Omega)$ denote the Hilbert space of square integrable functions on $\Omega$ (w.r.t. Lebesgue measure). The closed subspace of holomorphic square integrable functions will be denoted by $H^2(\Omega)$. This is the \emph{Bergman-space} associated with $\Omega$.
%
%

For each $i \in \{1,\ldots,n\}$ we have the \emph{Toeplitz operator} $T_{z_i} \in \sL(H^2(\Omega))$ given by multiplication with the $i^{\T{th}}$ coordinate function $z_i : \Omega \to \cc$. We will then focus on the commuting tuple of Toeplitz operator $T_z := (T_{z_1},\ldots,T_{z_n})$. The next result can be found as \cite[Theorem 1.3]{SaShUp:TPF}.

\begin{theorem}\label{t:pseudo}
Suppose that $\Omega$ is pseudoconvex and that $\pa \Omega = \pa( \T{cl}(\Omega))$. Then the following assertions are valid:
\begin{enumerate}
\item The spectrum for $T_z$ is the closure of $\Omega$, $\T{Sp}(T_z) = \T{cl}(\Omega)$.
\item The essential spectrum for $T_z$ is the boundary of $\Omega$, $\T{Sp}_{\T{ess}}(T_z) = \pa \Omega$.
\item For each $\la \in \Omega$ the homology of $T_z -\la$ is concentrated in degree $0$ and $\T{Ind}(T_z - \la) = - 1$.
\end{enumerate}
\end{theorem}

Let $g : \T{Sp}(T_z) = \T{cl}(\Omega) \to \cc^n$ be a holomorphic function. We then have the commuting tuple of Toeplitz operators $T_g = (T_{g_1},\ldots,T_{g_n})$ where each $T_{g_i} \in \sL(H^2(\Omega))$ acts by multiplication with $g_i$.
%

As an application of our local index theorem we can now obtain the following.

\begin{cor}
Suppose that $\Omega$ is pseudoconvex and that $\pa \Omega = \pa( \T{cl}(\Omega))$. Suppose furthermore that $0 \notin g(\pa \Omega)$ and that $\la \in Z(g)$. Then the commuting tuple $T_g$ is Fredholm and the local index at $\la$ is given by
\[
\T{Ind}_\la(T_g) = - \T{deg}_\la(g).
\]
In particular we have that $\T{Ind}(T_g) = - \T{deg}(0,g,\Om)$.
\end{cor}
\begin{proof}
By Theorem \ref{t:localind} and Theorem \ref{t:pseudo} we only need to prove that $T_f = f(T_z)$ for each holomorphic function $f : \T{cl}(\Om) \to \cc$. Let $\mu \in \Om$ and let $\ep_\mu : H^2(\Omega) \to \cc$ denote the functional given by evaluation at $\mu$. Let $\xi \in H^2(\Omega)$. It is then sufficient to show that $\ep_\mu\big( f(T_z)(\xi) \big) = f(\mu)\cd \ep_\mu(\xi)$. But this follows immediately from Proposition \ref{p:functoriality} and \cite[Theorem 3.16]{Tay:ACS}.
\end{proof}
%

\subsection{Hardy space}
In this section we will show that our local index theorem also applies to the Hardy space over the polydisc. Let us briefly recall some definitions. For general information we refer to \cite{Rud:FTP}.

We let $\cc[z_1,\ldots,z_n]$ denote the $\cc$-algebra of polynomials in the coordinate functions $z_1,\ldots,z_n : \B T^n \to \cc$ on the $n$-torus. The Hardy space over the polydisc is then defined as the closure of $\cc[z_1,\ldots,z_n]$ inside the Hilbert space $L^2(\B T^n)$ of square integrable functions on the $n$-torus. We will denote this Hardy space by $H^2(\B T^n)$.

The coordinate functions $z_1,\ldots,z_n$ act by multiplication on $H^2(\B T^n)$ giving rise to a commuting tuple $T_z := (T_{z_1},\ldots,T_{z_n})$ of Toeplitz operators. The next theorem can easily be deduced from \cite[Theorem 5]{Cur:FOD}:

\begin{theorem}\label{t:hardy}
The spectrum of $T_z$ is the closed polydisc, $\T{Sp}(T_z) = \B D^n$ whereas the essential spectrum of $T_z$ is the boundary, $\T{Sp}_{\T{ess}}(T_z) = \pa \B D^n$. For each point $\la \in U^n$ the homology of $T_z - \la$ is concentrated in degree $0$ and $\T{Ind}(T_z - \la) = - 1$.
\end{theorem}

Let $g : \T{Sp}(T_z) = \B D^n \to \cc^n$ be a holomorphic function. We then have the commuting tuple of Toeplitz operators $T_g = (T_{g_1},\ldots,T_{g_n})$ where $T_{g_i}$ acts by multiplication with $g_i$ on $H^2(\B T^n)$. As in the case of the Bergman spaces it is not hard to see that $T_g = g(T_z)$ where $g(T_z)$ is constructed using the analytic functional calculus. An application of our local index theorem now yields the following.

\begin{cor}
Suppose that $0 \notin g(\pa \B D^n)$ and that $\la \in Z(g)$. Then the commuting tuple $T_g$ is Fredholm and the local index at $\la$ is given by
\[
\T{Ind}_\la(T_g) = - \T{deg}_\la(g).
\]
In particular we have that $\T{Ind}(T_g) = - \T{deg}(0,g,U^n)$ where $U := \B D^\ci$ denotes the open disc.
\end{cor}

%
%
%
%
%

\bibliographystyle{amsalpha-lmp}

\begin{thebibliography}{\textsc{GrRe84}}

\bibitem[\textsc{CaPi99}]{CaPi:JTS}
\textsc{R.~Carey} and \textsc{J.~Pincus}, \emph{Joint torsion of {T}oeplitz
  operators with {$H\sp \infty$} symbols}, Integral Equations Operator Theory
  \textbf{33} (1999), no.~3, 273--304. \MR{1671481 (2000f:47050)}

\bibitem[\textsc{Cur81}]{Cur:FOD}
\textsc{R.~E. Curto}, \emph{Fredholm and invertible {$n$}-tuples of operators.
  {T}he deformation problem}, Trans. Amer. Math. Soc. \textbf{266} (1981),
  no.~1, 129--159. \MR{613789 (82g:47010)}

\bibitem[\textsc{EiLe77}]{EiLe:ADG}
\textsc{D.~Eisenbud} and \textsc{H.~I. Levine}, \emph{An algebraic formula for
  the degree of a {$C\sp{\infty }$} map germ}, Ann. of Math. (2) \textbf{106}
  (1977), no.~1, 19--44, With an appendix by Bernard Teissier, ``Sur une
  in{\'e}galit{\'e} {\`a} la Minkowski pour les multiplicit{\'e}s''.
  \MR{0467800 (57 \#7651)}

\bibitem[\textsc{EsPu96}]{EsPu:SDA}
\textsc{J.~Eschmeier} and \textsc{M.~Putinar}, \emph{Spectral decompositions
  and analytic sheaves}, London Mathematical Society Monographs. New Series,
  vol.~10, The Clarendon Press Oxford University Press, New York, 1996, Oxford
  Science Publications. \MR{1420618 (98h:47002)}

\bibitem[\textsc{Fa{\u\i}80}]{Fai:JSL}
\textsc{A.~S. Fa{\u\i}n{\v{s}}te{\u\i}n}, \emph{The joint essential spectrum of
  a family of linear operators}, Funktsional. Anal. i Prilozhen. \textbf{14}
  (1980), no.~2, 83--84. \MR{575225 (81f:47005)}

\bibitem[\textsc{Fan06}]{Fan:IPC}
\textsc{X.~Fang}, \emph{The {F}redholm index of a pair of commuting operators},
  Geom. Funct. Anal. \textbf{16} (2006), no.~2, 367--402. \MR{2231467
  (2007e:47010)}

\bibitem[\textsc{GrRe84}]{GrRe:CAS}
\textsc{H.~Grauert} and \textsc{R.~Remmert}, \emph{Coherent analytic sheaves},
  Grundlehren der Mathematischen Wissenschaften [Fundamental Principles of
  Mathematical Sciences], vol. 265, Springer-Verlag, Berlin, 1984. \MR{755331
  (86a:32001)}

\bibitem[\textsc{Kaa12}]{Kaa:JTS}
\textsc{J.~Kaad}, \emph{Joint torsion of several commuting operators}, Adv.
  Math. \textbf{229} (2012), no.~1, 442--486. \MR{2854180}

\bibitem[\textsc{Kod86}]{Kod:CDC}
\textsc{K.~Kodaira}, \emph{Complex manifolds and deformation of complex
  structures}, Grundlehren der Mathematischen Wissenschaften [Fundamental
  Principles of Mathematical Sciences], vol. 283, Springer-Verlag, New York,
  1986, Translated from the Japanese by Kazuo Akao, With an appendix by Daisuke
  Fujiwara. \MR{815922 (87d:32040)}

\bibitem[\textsc{Lev81}]{Lev:TSC}
\textsc{R.~N. Levy}, \emph{Notes on the taylor joint spectrum of commuting
  operators}, Banach Center Publ. VIII; Operator Theory (1981), 321--332.

\bibitem[\textsc{Lev89}]{Lev:FCO}
\bysame, \emph{Algebraic and topological {$K$}-functors of commuting
  {$n$}-tuple of operators}, J. Operator Theory \textbf{21} (1989), no.~2,
  219--253. \MR{1023314 (91g:47011)}

\bibitem[\textsc{Put80}]{Put:APF}
\textsc{M.~Putinar}, \emph{Algebraic properties of {F}redholm systems},
  Proceedings of the {F}ourth {C}onference on {O}perator {T}heory ({T}imi\c
  soara/{H}erculane, 1979) (Timi\c soara), Univ. Timi\c soara, 1980,
  pp.~187--198. \MR{657802 (84g:47013)}

\bibitem[\textsc{Put85}]{Put:BFI}
\bysame, \emph{Base change and the {F}redholm index}, Integral Equations
  Operator Theory \textbf{8} (1985), no.~5, 674--692. \MR{813356 (87j:47020)}

\bibitem[\textsc{Rud69}]{Rud:FTP}
\textsc{W.~Rudin}, \emph{Function theory in polydiscs}, W. A. Benjamin, Inc.,
  New York-Amsterdam, 1969. \MR{0255841 (41 \#501)}

\bibitem[\textsc{Sch69}]{Sch:NFA}
\textsc{J.~T. Schwartz}, \emph{Nonlinear functional analysis}, Gordon and
  Breach Science Publishers, New York, 1969, Notes by H. Fattorini, R.
  Nirenberg and H. Porta, with an additional chapter by Hermann Karcher, Notes
  on Mathematics and its Applications. \MR{0433481 (55 \#6457)}

\bibitem[\textsc{Ser00}]{Ser:LAM}
\textsc{J.-P. Serre}, \emph{Local algebra}, Springer Monographs in Mathematics,
  Springer-Verlag, Berlin, 2000, Translated from the French by CheeWhye Chin
  and revised by the author. \MR{1771925 (2001b:13001)}

\bibitem[\textsc{SSU89}]{SaShUp:TPF}
\textsc{N.~Salinas}, \textsc{A.~Sheu}, and \textsc{H.~Upmeier}, \emph{Toeplitz
  operators on pseudoconvex domains and foliation {$C\sp *$}-algebras}, Ann. of
  Math. (2) \textbf{130} (1989), no.~3, 531--565. \MR{1025166 (91e:47026)}

\bibitem[\textsc{Tay70a}]{Tay:ACS}
\textsc{J.~L. Taylor}, \emph{The analytic-functional calculus for several
  commuting operators}, Acta Math. \textbf{125} (1970), 1--38. \MR{0271741 (42
  \#6622)}

\bibitem[\textsc{Tay70b}]{Tay:JSC}
\bysame, \emph{A joint spectrum for several commuting operators}, J. Functional
  Analysis \textbf{6} (1970), 172--191. \MR{0268706 (42 \#3603)}

\bibitem[\textsc{Wei94}]{Wei:IHA}
\textsc{C.~A. Weibel}, \emph{An introduction to homological algebra}, Cambridge
  Studies in Advanced Mathematics, vol.~38, Cambridge University Press,
  Cambridge, 1994. \MR{1269324 (95f:18001)}

\end{thebibliography}

\def\cprime{$'$}
\providecommand{\bysame}{\leavevmode\hbox to3em{\hrulefill}\thinspace}
\providecommand{\MR}{\relax\ifhmode\unskip\space\fi MR }
\providecommand{\MRhref}[2]{%
  \href{http://www.ams.org/mathscinet-getitem?mr=#1}{#2}
}
\providecommand{\href}[2]{#2}

\end{document}